\documentclass{amsart}
\usepackage[utf8]{inputenc}
\usepackage[mathscr]{eucal}
\usepackage{amssymb, amsmath,array, amscd,amsfonts,amsthm}
\usepackage{url,enumerate,soul}
\usepackage[dvips]{graphicx}

\usepackage{hyperref}
\hypersetup{
colorlinks=true,
linkcolor=blue,
filecolor=magenta, 
urlcolor=cyan,
}
\usepackage{color}

\usepackage{geometry}
\usepackage{marginnote}

\usepackage{pgf,tikz,pgfplots} 
\pgfplotsset{compat=1.13} 

\usetikzlibrary{arrows,tikzmark,quotes}
\usepackage{tikz-cd}
\tikzset{
	labl/.style={anchor=north,rotate = 270 ,inner sep=.5mm}
}
\tikzset{
	symbol/.style={
		draw=none,
		every to/.append style={
			edge node={node [sloped, allow upside down, auto=false]{$#1$}}}
	}
}

\tikzcdset{
	labels={font=\everymath\expandafter{\the\everymath\displaystyle}},
}

\def\defi[#1]{{\bf{#1}}}

\input xy
\xyoption{all}

\newtheorem{thm}{Theorem}[section]

\newtheorem{prop}[thm]{Proposition}

\newtheorem{lemma}[thm]{Lemma}

\newtheorem{cor}[thm]{Corollary}
\newtheorem{claim}{Claim}

\newtheorem{innerthmintro}{Theorem}
\newenvironment{thmintro}[1]
{\renewcommand\theinnerthmintro{#1}\innerthmintro}
{\endinnerthmintro}

\theoremstyle{definition}
\newtheorem{definition}[thm]{Definition}
\newtheorem{remark}[thm]{Remark}

\numberwithin{equation}{section}

\newcommand{\Par}{\mathrm{Par}}

\newcommand{\Hig}{\mathrm{Higgs}}

\newcommand{\rk}{\mbox{rk}}

\newcommand{\elm}{\mathrm{elm}}
\newcommand{\ZZ}{\mathcal Z}
\newcommand{\BB}{\operatorname{Bun}_{0}}
\newcommand{\contotal}{\mathfrak{Con}^\nu}

\definecolor{grn}{rgb}{0.0, 0.49, 0.0}

\usepackage{amsaddr}

\begin{document}

\title{On the moduli of logarithmic connections on elliptic curves}

\author[T. Fassarella, F. Loray and A. Muniz]{Thiago Fassarella$^1$, Frank Loray$^2$ and Alan Muniz$^{1,3}$ }

\address{\small $^1$Universidade Federal Fluminense -- UFF, Rua Alexandre Moura 8, Niterói, RJ, Brazil }
\address{\small $^2$Univ Rennes, CNRS, IRMAR - UMR 6625, F-35000 Rennes, France }
\address{\small $^3$Universidade Federal do Espírito Santo -- UFES, Av. Fernando Ferrari 514, Vitória, ES, Brasil}
\email{tfassarella@id.uff.br, frank.loray@univ-rennes1.fr, alannmuniz@gmail.com}

\thanks{
The second author is supported by CNRS, and ANR-16-CE40-0008 project ``Foliage''. 
This work was conducted during the postdoctoral periods of the third author at UFES and UFF, when he was supported in part by the Coordenação de Aperfeiçoamento de Pessoal de Nível Superior -- Brasil (CAPES) -- Finance Code 001. The authors also thank Brazilian-French Network in Mathematics and CAPES-COFECUB project MA 932/19.
}
\date{\today}

\begin{abstract}
We describe moduli spaces of logarithmic rank $2$ connections on elliptic curves
with $n \geq 1$ poles and generic residues. In particular, we generalize a previous work by the first and second named authors. Our main approach is to analyze the underlying parabolic bundles; their stability and instability play a major role. 
\end{abstract}

\maketitle

\tableofcontents

\section{Introduction}

In this paper, we investigate the geometry of certain moduli spaces of connections on complex elliptic curves $C$. We will consider pairs $(E,\nabla)$ where $E\rightarrow C$ is a rank $2$ vector bundle and $\nabla\colon E\rightarrow E\otimes\Omega^1_C(D)$ is a logarithmic connection with (reduced) polar divisor $D=t_1+\cdots+t_n$ with $n\geq 1$. We also prescribe the following data:
\begin{itemize}
\item The eigenvalues $(\nu_i^+, \nu_i^-)$ of ${\rm Res}_{t_i}(\nabla)$, for each $i = 1, \dots, n$, such that:
\begin{enumerate}  
	\item $\nu_1^{\epsilon_1}+\cdots+\nu_n^{\epsilon_n} \notin \mathbb{Z}$, for any $\epsilon_i \in \{+,- \}$;
	\item and $\nu_i^+ \neq \nu_i^-$, for $i = 1, \dots, n$; 
\end{enumerate} 
\item A trace connection $(L,\zeta)$, i.e. $\det(E)= L$ and ${\rm tr}(\nabla) = \zeta$;
\end{itemize}
In particular $\sum_i\nu_i^+ +\nu_i^-=-\deg(L)$ which is called the Fuchs relation. Once we have fixed this data, we can define the moduli space $\contotal$ of those pairs $(E,\nabla)$ up to isomorphism. Due to Inaba's construction \cite{Inaba}, $\contotal$ is a smooth irreducible quasi-projective variety of dimension $2n$, equipped with an algebraic symplectic structure. These moduli spaces have been constructed much earlier in the analytic setting (see for instance
\cite{SimpsonHarmBun,NaHK}):
they are building blocks for the non-abelian Hodge correspondance. They turn out to be hyperkahler manifolds
diffeomorphic to tamely ramified Higgs bundle moduli spaces
(see also \cite{Iwasaki91,Nitsure,Konno}). We also note that these moduli spaces  occur in geometric Langlands correspondance (see \cite{Arinkin}), Riemann-Hilbert correspondance and isomonodromic deformations (see \cite{IIS}). 


We will consider the forgetful map $\pi\colon (E,\nabla)\mapsto (E,{\bf p})$ which associates to
a connection an underlying quasi-parabolic bundle. Given a choice of signs $\epsilon_i \in \{+, -\}$ for each $i = 1, \dots, n$,
the parabolic data ${\bf p}^{\epsilon}(\nabla)=(p_1^{\epsilon_1}(\nabla),\ldots,p_n^{\epsilon_n}(\nabla))$ consists
of the $\nu_i^{\epsilon_i}$-eigenspace $p_i^{\epsilon_i}\subset E\vert_{t_i}$ for ${\rm Res}_{t_i}(\nabla)$ at each pole; these are well-defined since $\nu_i^+ \neq \nu_i^-$. There exist $2^n$ underlying quasi-parabolic structures for each connection, according to the choice of $\epsilon=(\epsilon_1,\ldots,\epsilon_n)$ giving rise to $2^n$ forgetful maps $\pi_\epsilon\colon (E,\nabla)\mapsto (E,{\bf p}^{\epsilon}(\nabla))$.

In Section \ref{sect:parbun} we study the quasi-parabolic bundles $(E,{\bf p})$ over $(C,D)$ that admit a connection $\nabla$ with prescribed trace and eigenvalues, compatible with parabolic directions. The major difference from the case $n=2$ investigated in \cite{FL} is that, when $n$ is odd, there exist pairs $(E,\nabla)$ such that all the underlying quasi-parabolic bundles $(E,{\bf p}^{\epsilon}(\nabla))$ are not $\mu$-semistable for any choice of weights; it occurs for the item \ref{deE0} of Lemma \ref{lem:indec}. Following the study of stability, we describe a wall-crossing phenomenon in Lemma \ref{lem:morph} and Lemma \ref{lem:boundary}. 

In Section \ref{sect:logcon} we study the logarithmic connections. We investigate $\contotal$ via the forgetful map to an underlying quasi-parabolic structure. We are especially concerned with the $\mu$-stability of these quasi-parabolic bundles. It turns out that there exists an open subset of $\contotal$ where the underlying vector bundle is $E_1$, the unique indecomposable vector bundle of degree one, with given determinant.
We call this open subset  ${\rm Con}^{\nu}$ and consider the map
\begin{align*}
\Par \colon {\rm Con}^{\nu} &\longrightarrow {{(\mathbb{P}^1 \times \mathbb{P}^1)\times\cdots \times (\mathbb{P}^1 \times \mathbb{P}^1)} }\\
(E,\nabla)&\longmapsto \left(\ p_1^+(\nabla),p_1^-(\nabla),\ldots,p_n^+(\nabla),p_n^-(\nabla)\ \right)
\end{align*}
that associates to each connection all its residual eigenspaces (i.e. with respect to all eigenvalues $\nu_i^+$ and $\nu_i^-$). The image of $\Par$ is contained in $S^n$, where $S$ is the complement of the diagonal in $\mathbb{P}^1 \times \mathbb{P}^1$. We show that this map is an isomorphism (cf. Theorem \ref{thm:parconSn}).


We also study the open subset $\contotal_{st}	\subset \contotal$ formed by pairs $(E,\nabla)$ which admit a $\mu$-stable parabolic bundle $(E,{\bf p}^\epsilon(\nabla))$ for some $\epsilon$ and some weight vector $\mu$. We call 
\[
\ZZ_n = \contotal\setminus\contotal_{st}
\] 
its complement. We describe $\ZZ_n$ in Theorem \ref{thm:dessigma}, it is empty for $n$ even and has four irreducible components which are isomorphic to $\mathbb C^n$ for $n$ odd. Assuming $\nu_{i}^{+}-\nu_{i}^{-}\notin \{0, 1, -1\}$ for $i\in\{1, \cdots, n\}$, we see that $\contotal_{st}$ admits an open covering given by open subsets isomorphic to $S^n$. It leads to the following characterization of $\contotal_{st}$ (cf. Theorem \ref{thm:glue}):

\begin{thmintro}{A}\label{intro:glue} 
	Assume that $\nu_{i}^{+}-\nu_{i}^{-}\notin \{0, 1, -1\}$ for $i\in\{1, \cdots, n\}$. Then $\contotal_{st}$ is obtained by gluing a finite number of copies of $S^n$ via birational maps
	\[
	\Psi_{J,I}^{\delta,\epsilon}\colon S^n \dashrightarrow S^n.
	\]
	Moreover, if $\epsilon = \delta$, then $\Psi_{J,I}^{\delta,\epsilon}$ preserves the fibers of $\pi_\epsilon$. 
\end{thmintro}

In the remainder of our work, we do a finer analysis of the open subset ${\rm Con}^{\nu}$. Let ${\rm Bun}$ denote the moduli space of parabolic vector bundles whose underling vector bundle is $E_1$; it is isomorphic to $(\mathbb P^1)^n$. 
In Section \ref{sect:fuchssyst}, we describe the affine bundle ${\rm Con}^\nu \rightarrow {\rm Bun}$ via Fuchsian systems; this construction yields a vector bundle $\mathcal{E}$ whose projectivization compactifies ${\rm Con}^\nu$. Although $\mathbb{P}(\mathcal{E})$ does not depend on the eigenvalues, the boundary divisor is determined by $\nu_i = \nu_i^+ - \nu_i^-$, $i=1, \dots, n$ (cf. Theorem \ref{thm:extension}):

\begin{thmintro}{B}\label{intro:extension}
The moduli space ${\rm Con}^\nu$ has compactification  $\overline{{\rm Con}^{\nu}}=\mathbb{P}(\mathcal{E})$, where the boundary divisor is isomorphic to $\mathbb{P}\Hig$, the projectivization of the space of Higgs fields on $E_1$. Moreover, the inclusion $\mathbb{P}\Hig\hookrightarrow \mathbb{P}(\mathcal{E})$ is determined, up to automorphisms of $\mathbb{P}(\mathcal{E})$, by $\left(\nu_{1},\dots, \nu_{n}\right)$.
\end{thmintro}

In Section \ref{sect:symp}, we deal with the symplectic structure of the moduli space and compute the explicit expression in the main chart ${\rm Con}^{\nu}\simeq S^n$, see \eqref{eq:Darboux0}. We show that $\left(\nu_{1},\dots, \nu_{n}\right)$ is detected by the symplectic structure (cf. Corollary \ref{cor:symptorelli}): 

\begin{thmintro}{C}\label{thm:D}	
If there exists a fiber preserving symplectic isomorphism
	\[
	\begin{tikzcd}
		\left( {\rm Con}^{\nu},\omega\right) \arrow[d, "\pi_{+}"] \arrow[r,"\Phi", "\sim"'] & \left( {\rm Con}^{\tilde{\nu}},\tilde{\omega}\right) \arrow[d, "\pi_{+}"] \\ {\rm Bun} \arrow[r, "\phi", "\sim"']& {\rm Bun}
	\end{tikzcd}
	\]
	then there exists a permutation $\sigma$ of $n$ elements such that $\tilde{\nu}_k = \nu_{\sigma(k)}$ for every $k\in\{1,\cdots,n\}$.
\end{thmintro}

In Theorem \ref{thm:D}, we have a fixed underlying space of parabolic bundles, ${\rm Bun}$, and we can recover $\left(\nu_{1},\dots, \nu_{n}\right)$ from the symplectic structure. The main result of \cite{FJ} is a Torelli type result which asserts that the moduli space of parabolic bundles determines the punctured curve $(C,D)$. We wonder if the same is true for ${\rm Con}^{\nu}$, or $\contotal$. This belief is based on some results in the literature, see \cite{Seb, BN, BM}.

In Section \ref{sect:appmap}, we conclude the paper by studying the Apparent map. A global section of $E_1$ plays the role of a cyclic vector for connection $\nabla$, which yields a second order ODE on $C$; the map ${\rm App}$ assigns to $\nabla$ the apparent singular points of this equation, see \cite{LS}. It leads to an interesting result about the birational geometry of $\overline{ {\rm Con}^{\nu}}$ (Theorem \ref{thm:apparentbirat}):

\begin{thmintro}{D}
The map $\pi_+\times \operatorname{App}$ induces a birational map
\[
\pi_+\times \operatorname{App} \colon \overline{\rm{Con}^{\nu}} \dashrightarrow {\rm Bun}\times |\mathcal O_C(w_{\infty}+D)|
\]
whose indeterminacy locus is contained in $\overline{\rm{Con}^{\nu}} \backslash \rm{Con}^{\nu}$. Moreover, given $(E,{\bf{p}})\in {\rm Bun}$, the rank of
\[
\left.(\pi_+\times \operatorname{App})\right|_{\pi_+^{-1}(E,{\bf{p}})}\colon \pi_+^{-1}(E,{\bf{p}}) \longrightarrow { \{(E,{\bf{p}})\}\times} |\mathcal O_C(w_{\infty}+D)|
\]
coincides with the cardinality of the set $\{i\mid p_i\not\subset \mathcal O_{C} \}$.
\end{thmintro}

Logarithmic connections on elliptic curves have being investigated by several authors. First, they were studied in the works of Okamoto and later by Kawai in \cite{kawai2003} in relation with isomonodromic deformations. In particular, the symplectic form of \cite[Theorem 1]{kawai2003}, when restricted to a fixed punctured curve $(C,D)$, must coincide to that one we give in \eqref{eq:Darboux0}. However, as \cite{kawai2003} is dealing with analytic differential equations written in terms of Weierstrass zeta functions, it is not so easy to relate the computations with ours. More recently, our moduli space has been considered and proved to be related to some moduli spaces of logarithmic connections on $\mathbb P^1$ for the special cases $n=1$ (see \cite{Lame}),  and $n=2$ with special eigenvalues (see \cite{FrankValente}).
We expect that our explicit algebraic approach will allow us in a forthcoming work to relate
with Kawai's parameters and provide an algebraic expression of isomonodromy equation in our context.

\begin{remark}[Notation and convention]\label{rem:notation}
Throughout the text $C$ will denote a genus one curve and $D=t_1+\dots +t_n$ will be a reduced divisor on $C$.
Let $w_\infty \in C$ be a point and let $w_0,w_1$ and $w_\lambda$ be the torsion points of the elliptic curve $(C, w_\infty)$. In the construction of $\contotal$, we will assume (without loss of generality) that the determinant is $\mathcal{O}_C(w_\infty)$ and that $D+ w_0 + w_1+w_\lambda$ is reduced.
\end{remark} 

\subsection*{Aknowledgements.} We warmly thank the anonymous referee for many useful comments and suggestions.

\section{Parabolic vector bundles}\label{sect:parbun}
Let $C$ be an elliptic curve with $w_\infty \in C$ being its distinguished point. A rank two quasi-parabolic vector bundle $ (E, {\bf p})$ on $\left(C, D\right)$, $D=t_1+\dots +t_n$, consists of a holomorphic vector bundle $E$ of rank two on $C$ and a collection ${\bf p} = \{p_{1}, \dots , p_n\}$ of $1$-dimensional linear subspaces $p_{i} \subset E_{t_{i}}$. We refer to the points $t_i$ as parabolic points, and to the subspace $p_{i} \subset E_{t_{i}}$ as the parabolic direction of $E$ at $t_i$. 

A triple $(E,{\bf p}; \mu)$ of a quasi-parabolic vector bundle and an $n$-tuple $\mu = (\mu_1,\dots, \mu_ n)$ of real numbers in the interval $(0,1)$ is called parabolic vector bundle of rank two. We often write $(E,{\bf p})$ for a parabolic vector bundle when the choice of the weight $\mu$ is clear. 

Let $(E,{\bf p}; \mu)$ be a parabolic vector bundle and let $L\subset E$ be a line subbundle then we define 
\[
{\rm Stab}_{\mu}(L):= \deg E - 2 \deg L + \sum_{p_{k}\neq L_{t_{k}}} \mu_{k} - \sum_{p_{k}= L_{t_{k}}} \mu_{k}.
\]
We say that $(E,{\bf p}; \mu)$ is semistable if ${\rm Stab}_{\mu}(L)\geq 0$ holds for every $L\subset E$. It is stable if the strict inequality holds for every line subbundle $L\subset E$. 
We call ${\rm Stab}_{\mu}(L)$ the parabolic stability of $L\subset E$ with respect to $\mu$. 

We denote by $\operatorname{Bun}_{w_{\infty}}^{\mu}$ the moduli space of semistable parabolic vector bundles $(E,{\bf p}; \mu)$ on $(C,D)$ with $\det E = \mathcal O_C(w_{\infty})$. In this case, either $E\simeq L\oplus L^{-1}(w_{\infty})$ or $E\simeq E_{1}$, where $E_1$ is the unique non trivial extension 
\[
0 \longrightarrow \mathcal O_C \longrightarrow E_{1} \longrightarrow \mathcal O_C(w_{\infty}) \longrightarrow 0.
\]

If there exists $L\subset E$ such that ${\rm Stab}_{\mu}(L)$ is zero then the weights lie on the hyperplane 
\[
H(d,I):=\left\{\mu \mathrel{\Big|} 1-2d + \sum_{k\notin I} \mu_{k} - \sum_{k\in I}\mu_k = 0\right\}
\]
where $d=\deg L$ and $I\subset \{1,\dots , n\}$ denotes the set of indices of those parabolic directions $p_k\subset L_{t_k}$. A connected component of the complement in $(0,1)^n$ of all these hyperplanes $H(d,I)$ is called a chamber. If $\mu$ and $\tilde{\mu}$ belong to the same chamber then $\operatorname{Bun}_{w_{\infty}}^{\mu}=\operatorname{Bun}_{w_{\infty}}^{\tilde{\mu}}$, see for example \cite{MS} or \cite[Lemma 2.7]{BH}.

In the next result, we define an interesting chamber; the underlying vector bundle is fixed and the corresponding moduli space is a product of projective lines.

\begin{prop}\label{prop:1stchamb}The following assertions hold:
\begin{enumerate}
\item\label{Fstat1} The set $\mathfrak{C}:=\{\mu\in (0,1)^n \mid \sum_{k=1}^n \mu_k < 1\}$ is a chamber. 
\item\label{Fstat2} If $\mu\in \mathfrak{C}$ then $\operatorname{Bun}_{w_{\infty}}^{\mu} = \{ (E,{\bf p})\mid E=E_1\}$. Moreover, it is isomorphic to $(\mathbb P^1)^n$.
\end{enumerate}
\end{prop}

\begin{proof}
Note that $\mathfrak{C}$ is convex, hence connected. Then \eqref{Fstat1} follows from proving that $\mathfrak{C}$ does not intersect any wall $H(d,I)$. This is straightforward and we leave it to the reader.

Now we prove \eqref{Fstat2}. Recall that $\det E = \mathcal{O}_C(w_\infty)$ implies that either $E=E_1$ or $E= L\oplus L^{-1}(w_\infty)$ with $\deg L \geq 1$. For the later we have
\[
{\rm Stab}_{\mu}(L) = 1-2\deg L + \sum_{p_{k}\neq L_{t_{k}}} \mu_{k} - \sum_{p_{k}= L_{t_{k}}} \mu_{k} \leq 1-2\deg L + \sum_{k=1}^n \mu_k  < 2-2\deg L \leq 0
\]
for any $\mu \in \mathfrak{C}$, hence $E$ cannot be $\mu$-semistable. Hence $E = E_1$ and each parabolic bundle is completely determined by
\[
(p_1,\dots, p_n) \in \mathbb{P}\left(E_1|_{t_1}\right)\times \dots \times \mathbb{P}\left(E_1|_{t_n}\right) \simeq (\mathbb P^1)^n.
\]
Thus we get the desired isomorphism.
\end{proof}

For a weight vector $\mu = (\mu_{1}, \dots, \mu_{n})\in (0,1)^n$ and a subset $I\subset \{1, \dots, n\}$ of even cardinality, 
we consider the map $\varphi_I\colon (0,1)^n\longrightarrow(0,1)^n$ defined by
\[
\varphi_I(\mu) := (\mu'_{1}, \dots, \mu'_{n})\in (0,1)^n
\]
where $\mu'_i=\mu_i$ if $i\not\in I$, and $\mu'_i=1-\mu_i$ if $i\in I$. Note that $\varphi_I$ is continuous and preserves the walls $H(d,J)$. Then the image of $\mathfrak{C}$ by $\varphi_I$ yields a new chamber 
\begin{eqnarray}\label{CI}
{\mathfrak{C}}_I:=\left\{\mu\in (0,1)^n \mathrel{\Big|} \sum_{k\notin I} \mu_{k} - \sum_{k\in I} \mu_k + |I|<1\right\}
\end{eqnarray}
where $|I|$ is the cardinality of $I$. When $I=\emptyset$ then ${\mathfrak{C}}_I = \mathfrak{C}$.

Each $\varphi_I$ admits a modular realization as an elementary transformation, which we now describe. Consider the following exact sequence of sheaves
\[
0 \longrightarrow E' \stackrel{\alpha}{\longrightarrow} E \stackrel{\beta}{\longrightarrow} \bigoplus_{i\in I}(E_{t_i}/p_i)\longrightarrow 0 
\]
where for each (local) section $s$ of $E$ we define $\beta(s) = (\beta_1(s), \dots, \beta_n(s))$ by $\beta_j(s) = s(t_j) \pmod{p_j}$ if $s$ is defined at $t_j$, and $\beta_j(s)=0$ otherwise.
Then $E'$ is a vector bundle of rank two such that
\[
\det E' = \det E \otimes \mathcal O_{C}\left(-\sum_{i\in I}t_i\right).
\]
In particular, $E'$ has degree $1-|I|$. We define a natural quasi-parabolic structure for $E'$ as follows. If $i\not\in I$ then $\alpha_{t_i}\colon E'_{t_i}\longrightarrow E_{t_i}$ is an isomorphism and 
\[
p_i'=(\alpha_{t_i})^{-1} (p_i)\subset E'_{t_i}
\]
is the parabolic direction at $t_i$. If $i\in I$ we define $p_i'=\ker(\alpha_{t_i})$ as the parabolic direction at $t_i$. 
This operation corresponds to the birational transformation of ruled surfaces $\mathbb P(E)\dashrightarrow \mathbb P(E')$ obtained
by blowing-up the points $p_i\in\mathbb{P}(E_{t_i})$ and then blowing-down the strict transforms of the fibers $\mathbb{P}(E_{t_i})$ to the points $p'_i\in\mathbb{P}(E'_{t_i})$, $i\in I$. This is well-defined since the $p_i$ lie on different fibers. 

\begin{figure}[ht]\label{fig:elemtransf}

\centering	
\definecolor{color_29791}{rgb}{0,0,0}
\definecolor{color_274846}{rgb}{1,0,0}
\begin{tikzpicture}[line cap=round,line join=round,>=triangle 45]
	\path(-5pt,-130pt) node {$\displaystyle\mathbb{P}(E)$};
	\path(300pt,-130pt) node {$\displaystyle\mathbb{P}(E')$};
	\draw[color_29791,line width=1.363616pt]
	(11.24959pt, -154.1211pt) -- (86.24838pt, -154.1211pt)
	;
	\draw[color_29791,line width=1.499977pt]
	(18.74946pt, -161.621pt) -- (18.74946pt, -94.12207pt)
	;
	\draw[color_29791,line width=1.499977pt]
	(33.74922pt, -161.621pt) -- (33.74922pt, -94.12207pt)
	;
	\draw[color_29791,line width=1.499977pt]
	(63.74874pt, -161.621pt) -- (63.74874pt, -94.12207pt)
	;
	\draw[color_29791,line width=1.499977pt]
	(78.74849pt, -161.621pt) -- (78.74849pt, -94.12207pt)
	;
	\draw[color_29791,line width=1.363616pt]
	(206.2464pt, -154.1211pt) -- (281.2452pt, -154.1211pt)
	;
	\draw[color_29791,line width=1.499977pt]
	(213.7463pt, -161.621pt) -- (213.7463pt, -94.12207pt)
	;
	\filldraw[color_274846][nonzero rule]
	(228.7461pt, -161.621pt) -- (228.7461pt, -94.12207pt)
	;
	\draw[color_274846,line width=1.499977pt]
	(228.7461pt, -161.621pt) -- (228.7461pt, -94.12207pt)
	;
	\draw[color_29791,line width=1.499977pt]
	(273.7453pt, -161.621pt) -- (273.7453pt, -94.12207pt)
	;
	\draw[color_29791,line width=1.363616pt]
	(108.748pt, -101.6219pt) -- (183.7468pt, -101.6219pt)
	;
	\draw[color_29791,line width=1.499977pt]
	(116.2479pt, -109.1218pt) -- (116.2479pt, -41.62292pt)
	;
	\draw[color_29791,line width=1.499977pt]
	(176.2469pt, -109.1218pt) -- (176.2469pt, -41.62292pt)
	;
	\draw[color_29791,line width=1.499977pt]
	(131.2477pt, -109.1218pt) .. controls (132.5732pt, -94.12207pt) and (137.7363pt, -79.12231pt) .. (123.7478pt, -64.12256pt)
	;
	\draw[color_274846,line width=1.499977pt]
	(132.0052pt, -41.62292pt) .. controls (136.2878pt, -81.14935pt) and (129.827pt, -72.80084pt) .. (123.7478pt, -79.12231pt)
	;
	\draw[color_274846,line width=0.749987pt]
	(29.99928pt, -120.3717pt) -- (37.49916pt, -127.8716pt)
	;
	\draw[color_274846,line width=0.749987pt]
	(37.49916pt, -120.3717pt) -- (29.99928pt, -127.8716pt)
	;
	\draw[color_29791,line width=0.749987pt]
	(224.9961pt, -150.3712pt) -- (232.496pt, -157.8711pt)
	;
	\draw[color_29791,line width=0.749987pt]
	(232.496pt, -150.3712pt) -- (224.9961pt, -157.8711pt)
	;
	\draw[color_29791,line width=0.717154pt,->]
	(104.9981pt, -67.8725pt) .. controls (100.4534pt, -67.8725pt) and (95.90799pt, -67.8725pt) .. (91.36235pt, -69.12244pt) .. controls (86.8167pt, -70.37238pt) and (82.27201pt, -72.87195pt) .. (78.29442pt, -75.99713pt) .. controls (74.31682pt, -79.12225pt) and (70.9084pt, -82.87152pt) .. (67.49869pt, -86.62219pt)
	;
	\draw[color_29791,line width=0.734581pt, ->]
	(185.5486pt, -67.8725pt) .. controls (190.3294pt, -67.8725pt) and (195.1109pt, -67.8725pt) .. (199.8927pt, -69.12244pt) .. controls (204.6745pt, -70.37238pt) and (209.4553pt, -72.87195pt) .. (213.6396pt, -75.99713pt) .. controls (217.8238pt, -79.12225pt) and (221.4093pt, -82.87152pt) .. (224.9962pt, -86.62219pt)
	;
	\draw[color_29791,line width=0.722823pt,dash pattern=on 3.1348383pt off 3.1348383pt, ->]
	(97.4982pt, -146.6212pt) -- (194.9966pt, -146.6212pt)
	;
	\draw[color_29791,line width=1.499977pt]
	(146.2474pt, -109.1218pt) .. controls (138.7475pt, -94.12207pt) and (153.7473pt, -90.37213pt) .. (142.4975pt, -64.12256pt)
	;
	\draw[color_274846,line width=1.499977pt]
	(147.005pt, -41.62292pt) .. controls (153.7473pt, -82.87225pt) and (146.2474pt, -75.37238pt) .. (142.4975pt, -86.62219pt)
	;
	\filldraw[color_274846][nonzero rule]
	(243.7458pt, -161.621pt) -- (243.7458pt, -94.12207pt)
	;
	\draw[color_274846,line width=1.499977pt]
	(243.7458pt, -161.621pt) -- (243.7458pt, -94.12207pt)
	;
	\draw[color_29791,line width=0.749987pt]
	(239.9959pt, -150.3712pt)-- (247.4958pt, -157.8711pt)
	;
	\draw[color_29791,line width=0.749987pt]
	(247.4958pt, -150.3712pt) -- (239.9959pt, -157.8711pt)
	;
	\draw[color_29791,line width=1.499977pt]
	(48.74899pt, -161.621pt) -- (48.74899pt, -94.12207pt)
	;
	\draw[color_274846,line width=0.749987pt]
	(44.99904pt, -135.3715pt) -- (52.49892pt, -142.8713pt)
	;
	\draw[color_274846,line width=0.749987pt]
	(52.49892pt, -135.3715pt) -- (44.99904pt, -142.8713pt)
	;
	\draw[color_274846,line width=0.749987pt]
	(59.99879pt, -150.3712pt) -- (67.49867pt, -157.8711pt)
	;
	\draw[color_274846,line width=0.749987pt]
	(67.49867pt, -150.3712pt) -- (59.99879pt, -157.8711pt)
	;
	\draw[color_29791,line width=0.749987pt]
	(254.9956pt, -127.8715pt) -- (262.4955pt, -135.3714pt)
	;
	\draw[color_29791,line width=0.749987pt]
	(262.4955pt, -127.8715pt) -- (254.9956pt, -135.3714pt)
	;
	\filldraw[color_274846][nonzero rule]
	(258.7456pt, -161.621pt) -- (258.7456pt, -94.12207pt)
	;
	\draw[color_274846,line width=1.499977pt]
	(258.7456pt, -161.621pt) -- (258.7456pt, -94.12207pt)
	;
	\draw[color_29791,line width=1.499977pt]
	(164.9971pt, -41.62292pt) .. controls (172.497pt, -56.62268pt) and (157.4972pt, -60.37262pt) .. (168.7471pt, -86.62219pt)
	;
	\draw[color_274846,line width=1.499977pt]
	(164.2395pt, -109.1218pt) .. controls (157.4972pt, -67.8725pt) and (164.9971pt, -75.37238pt) .. (168.747pt, -64.12256pt)
	;
\end{tikzpicture}
	\caption{Elementary transformation}
\end{figure}
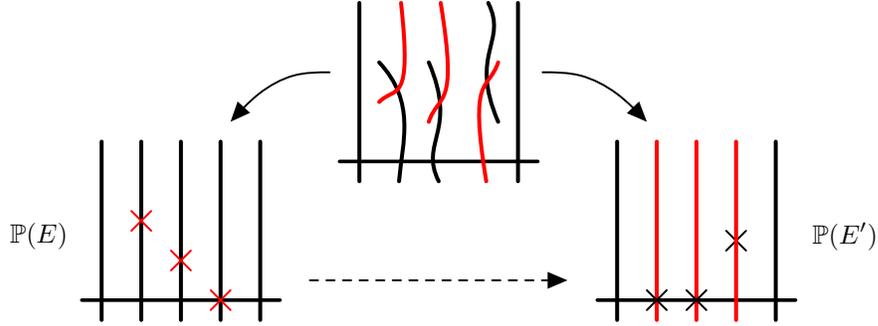

Since $|I|$ is even, we can fix a square root $L_0$ of the line bundle $\mathcal O_C\left(\sum_{i\in I}t_i\right)$, i.e.
\[
L_0^2 = \mathcal O_C\!\left(\sum_{i\in I}t_i\right). 
\] 
This gives a correspondence 
\[
\elm_I\colon \left(E, {\bf p}\right) \longmapsto \left(E'\otimes L_0, {\bf p'}\right)
\]
between quasi-parabolic vector bundles on $\left(C,D\right)$ which have $\mathcal O_C(w_{\infty})$ as determinant line bundle.

The reader can check that if $\left(E, {\bf p}\right)$ is semistable with respect to $\mu$, then $ \elm_I\left(E, {\bf p}\right)$ is semistable with respect to $\varphi_I(\mu)$. 
We conclude that the correspondence $\elm_I$ defines an isomorphism between moduli spaces 
\[
\elm_I\colon\operatorname{Bun}_{w_{\infty}}^{\mu} \longrightarrow \operatorname{Bun}_{w_{\infty}}^{\varphi_I(\mu)}.
\]

\begin{definition}\label{def:bunI}
Given $I\subset \{1, \dots, n\}$ of even cardinality, let $\mu\in \mathfrak{C}_I$. We will denote 
\[
{\rm Bun}^I = \operatorname{Bun}_{w_{\infty}}^{\mu}.
\]
When $I$ is the empty set, we write simply ${\rm Bun}$ instead of ${\rm Bun}^{\emptyset}$; it corresponds to the moduli space of parabolic vector bundles whose underlying vector bundle is $E_1$.
\end{definition}

\begin{remark}\label{rem:otherchamb}
From Proposition \ref{prop:1stchamb}, we conclude that ${\rm Bun}^I\simeq (\mathbb P^1)^n$ for any $I\subset \{1, \dots, n\}$ of even cardinality. 
\end{remark}

A quasi-parabolic vector bundle $(E, {\bf p})$ is called decomposable if there exist $(L, {\bf p'})$ and $(M, {\bf p''})$ such that $(E, {\bf p}) \simeq (L, {\bf p'})\oplus (M, {\bf p''})$ as quasi-parabolic vector bundles. Otherwise it is called indecomposable. Note that $(E, {\bf p})$ can be indecomposable with $E$ decomposable as a vector bundle.

\begin{lemma}\label{lem:indec}
Let $(E, {\bf p})$ be a rank two indecomposable quasi-parabolic bundle, over $(C,D)$, with $\det E = \mathcal{O}_C(w_\infty)$. Then one of the following holds:
\begin{enumerate}
\item $E$ is indecomposable, i.e. $E = E_1$;
\item $E = L \oplus L^{-1}(w_\infty)$ and $2 \leq 2 \deg L \leq n$;
\item\label{deE0} $E = L \oplus L^{-1}(w_\infty)$ with $L^2 = \mathcal{O}_C(D + w_\infty)$, hence $2 \deg L = n+1$. Moreover, every parabolic direction lies on $L^{-1}(w_\infty)$ except for one that lies outside both subbundles.
\end{enumerate}
\end{lemma}

\begin{proof}
When $E=E_1$ we have nothing to prove. So suppose that $E = L \oplus L^{-1}(w_\infty)$. Since $L\oplus L^{-1}(w_{\infty})\simeq M\oplus M^{-1}(w_{\infty})$ with $M = L^{-1}(w_{\infty})$ we can assume $\deg L = s\ge 1$.

To decompose $(E, {\bf p})$ we need to find an embedding of $L^{-1}(w_{\infty})$ in $E$ passing through every direction that does not lie on $L$. Note that this is the same as finding an automorphism of $E$ that sends every direction outside $L$ to $(0:1)$. Let $p_j = (u_j: 1)$ denote the parabolic direction over $t_j$ which is outside $L$. Recall that 
\[
{\rm End}(E) = \left\{\begin{pmatrix}
\alpha & \beta \\ 0 & \delta
\end{pmatrix} \mathrel{\Big|} \alpha, \delta \in \mathbb{C} , \, \beta \in {\rm H^0}(C, L^2(-w_\infty))\right\}.
\]
If $2s \geq n+2$ then $h^0(L^2(-w_\infty -D+t_j)) = 2s -n \geq 2$ and we are free to choose $\beta_j$ that vanishes on $t_i$ for $i\neq j$ and such that $\beta_j(t_j) = -u_j$. Thus, choosing $\beta = \sum_{j=1}^n \beta_j$, $\alpha=1$ and $\delta=1$, the corresponding automorphism sends any direction $p_j$ outside $L$ to $(0:1)$. 

Now set $2s = n+1$. By the same argument as above, to show that $(E,{\bf p})$ is decomposable, we need to find a section $\beta_j$ of ${\rm H^0}(C, L^2(-w_\infty))$ that vanishes on $t_i$ for $i\neq j$ and such that $\beta_j(t_j) = -u_j$, for each $j\in\{1,\dots, n\}$. We can find $\beta_j$ as required if $L^2(-w_\infty -D) \neq \mathcal{O}_C$. Indeed, assume $L^2(-w_\infty -D+t_j)\simeq \mathcal O_C(x_j)$ with $x_j\neq t_j$, and take any section $\alpha_j$ of $L^2(-w_\infty -D+t_j)$ with $\alpha_j(t_j)\neq 0$. The desired section is defined as $\beta_j = -\frac{u_j}{\alpha_j(t_j)}\alpha_j$. Hence $(E,{\bf p})$ is decomposable when $L^2(-w_\infty -D) \neq \mathcal{O}_C$. 

If $L^2 = \mathcal{O}_C(D + w_\infty)$ we can apply the same argument for $D-t_1$ instead of $D$ to find an embedding of $L^{-1}(w_{\infty})$ passing through $n-1$ parabolic directions outside $L$. In particular, if $(E,{\bf p})$ is indecomposable then there exists no parabolic direction on $L$ and this finishes the proof. 
\end{proof}

\begin{remark}\label{rem:unstparbun}
The parabolic bundles in the third case of Lemma \ref{lem:indec} have a peculiar property: they are never $\mu$-semistable, whatever is $\mu$. Indeed, for $E = L\oplus L^{-1}(w_{\infty})$ with $2\deg L =n+1$, no parabolic direction lying on $L$, and any weight $\mu$, we have
\[
{\rm Stab}_{\mu}(L)= -n + \sum_{j=1}^n \mu_j < 0.
\]
\end{remark}

We can give a partial converse to this fact. Note that if $E = E_1$ then any quasi-parabolic bundle is stable for $\mu\in \mathfrak{C}$.

\begin{lemma}\label{lem:stparbun}
Let $(E, {\bf p})$ be an indecomposable quasi-parabolic bundle such that $E = L\oplus L^{-1}(w_\infty)$ and $p_j$ lies outside $L$, for every $j$. If $2 \leq 2\deg L \leq n$ then there exists $I \subset \{1, \dots, n\}$, with $|I| =2\deg L$, such that $(E, {\bf p}) \in {\rm Bun}^I$.
\end{lemma} 

\begin{proof}
Given that every direction $p_j$ lies outside $L$, we may find an embedding of $L^{-1}(w_{\infty})$ that passes through some of these directions. Any subset of directions with cardinality $2\deg L-1$ admits at most one embedding of $L^{-1}(w_{\infty})$ passing through them. Since $(E, {\bf p})$ is indecomposable, such embedding cannot pass through all $p_j$. In particular, we can find $k \in \{ 2\deg L, \dots, n\} $ such that no embedding of $L^{-1}(w_{\infty})$ passes through the directions indexed by $I = \{ 1, \dots,2\deg L-1 , k\} $. It is straightforward to verify that $(E, {\bf p}) \in {\rm Bun}^I$.
\end{proof}

Until now we have only considered a rank two $E$ and its line subbundles $L \subset E$. But a more general setting will be suitable for the next results; we may allow subsheaves that are not saturated. We will consider general morphisms $L \rightarrow E$ that do not, necessarily, lead to an embedding of $L$ in $E$. Recall that, over a curve, being a subbundle means that there exist an injective morphism $L \hookrightarrow E$ whose cokernel is also a line bundle, i.e. $L$ is a saturated subsheaf of $E$. For a general morphism $\phi\colon L \rightarrow E$ this does not need to be true. However, we can factor out a divisor $Z$ where $\phi$ vanishes, leading to an injective morphism $L(Z) \hookrightarrow E$. For details, see \cite[Chapter 2, Proposition 5]{FRIED}.
On the other hand, given a subbundle $L \subset E$ we can produce a morphism $L(-Z) \rightarrow E$ that vanishes on the fibers over the support of $Z$.

Given a morphism $\phi\colon L\rightarrow E$ we say that its image passes through $p_j\subset E_{t_j}$ if $\phi_{t_j}(L_{t_j})\subset p_j$. 
\begin{lemma}\label{lem:morph}
Let $I\subset \{1,\dots, n\}$ have cardinality $2k+2$ with $k\ge 0$ and fix $\mu\in \mathfrak{C}_I$. Then $(E_1,{\bf p})$ is not $\mu$-semistable if and only if there exists a line bundle $L$ of degree $\deg L=-k$ and a morphism $L\rightarrow E_1$ whose image passes through $p_j$ for all $j\in I$. 
\end{lemma}

\begin{proof}
Fix $\mu = (\mu_1, \cdots, \mu_n)\in \mathfrak{C}_I$ and recall that $\mu = \varphi_I(\mu')$ for some $\mu' \in \mathfrak{C}$. First assume that $(E_1,{\bf p})$ is not $\mu$-semistable and let $M\subset E_1$ be a subbundle such that $Stab_{\mu}(M)<0$. Denote  $A = \{j \in  \{1, \cdots, n\} \mid p_{j}= M_{t_{j}} \}$ so that
\begin{align*}
   0 > {\rm Stab}_{\mu}(M) & = 1 - 2 \deg M + \sum_{\substack{j \notin A \\ j\in I }} \mu_{j} + \sum_{\substack{j \notin A \\ j\notin I }} \mu_{j} - \sum_{\substack{j \in A \\ j \in I}} \mu_{j}  - \sum_{\substack{j \in A \\ j \notin I}} \mu_{j} \\
    & = 1 - 2 \deg M + \sum_{\substack{j \notin A \\ j\in I }} (1-\mu_{j}') + \sum_{\substack{j \notin A \\ j\notin I }} \mu_{j}' - \sum_{\substack{j \in A \\ j \in I}} (1-\mu_{j}')  - \sum_{\substack{j \in A \\ j \notin I}} \mu_{j}' \\
    & = 1 - 2 \deg M + |I|-2|A\cap I| - \sum_{\substack{j \notin A \\ j\in I }} \mu_{j}' + \sum_{\substack{j \notin A \\ j\notin I }} \mu_{j}' + \sum_{\substack{j \in A \\ j \in I}} \mu_{j}' - \sum_{\substack{j \in A \\ j \notin I}} \mu_{j}'.
\end{align*}
From $\sum_j \mu_j' <1$ and $|I| = 2(k+1)$ we get that 
\[
- \deg M + k+1 - |A\cap I| \leq -1.
\]
Now let $u := \deg M + k$. The inequality above plus $\deg M \leq 0$ implies that 
\[
| I \setminus (I\cap A)| = 2k+ 2 - |A\cap I| \leq u \leq k \leq |I|,
\]
hence there exists $J \subset \{1, \cdots, n\}$ such that $|J| = u$ and $I \setminus (I\cap A) \subset J \subset I$. Define $Z = \sum_{j \in  J} t_j$. Considering the inclusions $M \hookrightarrow E_1$ and $\mathcal{O}_C(-Z)\hookrightarrow \mathcal{O}_C$ we define a map by the composition
\[
\phi\colon M \otimes \mathcal{O}_C(-Z) \longrightarrow E_1(-Z) \longrightarrow E_1,
\]
with the property that it gives the same directions over $D-Z$ and vanishes over $Z$. Hence the image of $\phi$ passes through every direction from $I$ and $L := M(-Z)$ is our desired line bundle.
 
Conversely, suppose that there exists a degree $-k$ line bundle $L$ and a nontrivial morphism $\phi \colon L \rightarrow E_1$ passing through every $p_j$, $j\in I$. Let $Z$ be the zero divisor of $\phi$ and consider the reduction $\phi' \colon L(Z) \rightarrow E_1$. Then $\phi'$ realizes $L(Z)$ as a subbundle of $E_1$ and, in particular, $\deg Z \leq k$. 

On the other hand, we have that $p_j$ lie on $L(Z)$ for every $j$ such that $t_j \not\in {\rm Supp } Z$. If $A$ is given as above, we have $|A\cap I| \geq 2k+2 - \deg Z_{red}$, hence
\begin{align*}
    1 - 2\deg L(Z) + |I| -2|A\cap I| & = 4k+ 3 -2|A\cap I| - 2\deg Z \\
    & \leq 2(\deg Z_{red}-\deg Z) - 1 \leq -1,
\end{align*}
which implies ${\rm Stab}_{\mu}(L(Z)) <0$ for $\mu\in \mathfrak{C}_I$.
\end{proof}

We now see the real advantage of switching to this slightly more general setting. The previous lemma describes a wall-crossing phenomenon. And the next lemma can be used to describe geometrically the space of quasi-parabolic bundles that become unstable when we cross a wall. 

\begin{lemma}\label{lem:boundary}
Let $n = 2k+2$ for some $k\geq 0$. Let $V\subset (\mathbb{P}^1)^n$ be the locus of points that correspond to quasi-parabolic bundles $(E_1, {\bf p})$ satisfying the following property: there exist a line bundle $L$ of degree $-k$ and a morphism $\phi \colon L \longrightarrow E_1$ whose image passes through ${\bf p}$. Then $V$ is a hypersurface of degree $(2, \dots, 2)$.
\end{lemma}

\begin{proof}
Let $\pi_j \colon (\mathbb{P}^1)^n \longrightarrow (\mathbb{P}^1)^{n-1}$ be the projection given by forgetting the $j$th component and let $h_j$ be the class of a fiber of $\pi_j$. Then we only need to show that $V \cap h_{j} = 2$ for every $j$. Up to permuting indices we only need to consider $j= 2k+2$.

If $k=0$ the result follows from \cite[Proposition 3.3]{Nestor}. Indeed, for each degree $0$ line bundle $L\in {\rm Pic}^{0}(C)$ there exists a unique map $\phi\colon L \rightarrow E_1$ and the map $L \mapsto \phi_{t_1}(L) \in \mathbb{P}^1$ is generically $2:1$. Then, for a generic direction $p_1$, there exist two choices for $L\in {\rm Pic}^{0}(C)$ such that $\phi_{t_1}(L)\subset p_1$. Therefore, $(p_1,p_2)\in V$ if and only if $p_2$ is one of the directions defined by these line bundles, i.e. $V\cap h_2 = 2$.

Now we consider $k\geq 1$. We will show that we can reduce to the previous case. Fix $p_1, \dots, p_{2k}$ generic directions. By generic we mean that there exists no subbundle of degree at least $1-k$ passing through these directions. Let $L\in {\rm Pic}^{-k}(C)$ be any line bundle. To give a map $\phi \colon L\longrightarrow E_1$ passing through $p_1, \dots, p_{2k}$ is equivalent to giving a map $L\longrightarrow E'$, where $E'$ is obtained by elementary transformation with respect to $p_1, \dots, p_{2k}$. Indeed, we have
\[
0 \longrightarrow E' \stackrel{\alpha}{\longrightarrow} E_1 \stackrel{\beta}{\longrightarrow} \bigoplus_{j=1}^{2k} (E_1)_{t_j}/p_j \longrightarrow 0
\]
and $\beta \circ \phi =0$ if and only if there exists $\phi' \colon L \longrightarrow E'$ such that $\phi = \alpha \circ \phi'$. Nonetheless, this is equivalent to giving a map 
\[
\phi' \otimes 1 \colon L \otimes M \longrightarrow E'\otimes M
\]
where $M$ is a line bundle such that $M^2 = \mathcal{O}_C(t_1 + \dots + t_{2k})$. 

Since $p_1, \dots, p_{2k}$ are generic, $E'$ is indecomposable. In particular, $E'\otimes M= E_1$. We then apply the same argument of the case $k=0$ to the directions $p_{2k+1}$ and $p_{2k}$ to show that $V\cap h_{2k+2} = 2$.
\end{proof}

\begin{definition}\label{def:gammaI}
Let $n\geq 2$ be an integer and let $I \subset \{1, \dots, n \}$ be a subset of even cardinality.
We will denote by $\Gamma_I\subset (\mathbb{P}^1)^n$ the subvariety that parameterizes quasi-parabolic bundles $(E_1, {\bf p})$ that are not $\mu$-semistable for $\mu \in \mathfrak{C}_I$.
\end{definition}

\begin{cor}\label{cor:dimgammaI}
The subvariety $\Gamma_I\subset (\mathbb{P}^1)^n$ is a hypersurface of degree $(d_1, \dots, d_n)$, where $d_i = 2$ if $i\in I$ and $d_i =0$ otherwise.
\end{cor}

\begin{proof}
Note that the formation of $\Gamma_I$ depends only on the directions indexed by $I$. Then it will be a product $\Gamma_I \simeq V \times(\mathbb{P}^1)^{n-|I|}$. Therefore, we may reduce to the case $\Gamma_I = V$, i.e. $|I| =n$ and the conclusion follows from Lemma \ref{lem:morph} and Lemma \ref{lem:boundary}.
\end{proof}

\begin{remark}\label{rem:elmGammaI}
A quasi-parabolic bundle $(E_1,{\bf p})$ is not $\mu$-semistable for $\mu\in \mathfrak{C}_I$ if and only if $\elm_I(E_1,{\bf p}) = (E, {\bf p'})$ is not $\mu'$-semistable for $\mu' =\varphi_I(\mu) \in \mathfrak{C}$. The later occurs if and only if $E$ splits. Therefore $\Gamma_I$ corresponds, via $\elm_I$, to the locus in $\operatorname{Bun}^I$ of those quasi-parabolic bundles whose underlying vector bundles split.
\end{remark}

\section{Logarithmic connections}\label{sect:logcon}
A logarithmic connection on a rank two vector bundle $E$ over $C$ with polar divisor $D=t_1+\cdots +t_n$ is a $\mathbb C$-linear map
\[
\nabla\colon E \longrightarrow E\otimes \Omega_C^1(D)
\]
satisfying the Leibniz rule
\[
\nabla(fs)= s \otimes df + f \nabla (s)
\]
for (local) sections $s$ of $E$ and $f$ of $\mathcal{O}_C$. 
If $t\in C$ is a pole for $\nabla$ and $U\subset C$ is a small trivializing neighborhood of $t$, we write $\nabla|_U = d + A$ where $d\colon \mathcal O_C \longrightarrow \Omega_C^1$ is the exterior derivative and $A$ is a $2\times 2$ matrix whose coefficients are $1$-forms with at most simple poles on $t$. 
Note that $A$ depends on the trivialization, but its similarity class does not. Then the residue endomorphism 
\[
{\rm Res}_{t}(\nabla):={\rm Res}_{t}(A) \in {\rm End}(E_{t})
\] 
is well defined. Let $\nu_k^+$ and $\nu_k^-$ be the eigenvalues of ${\rm Res}_{t_k}(\nabla)$. The data 
\[
\nu=(\nu_1^{+},\nu_{1}^-,...,\nu_n^{+},\nu_n^-) \in \mathbb C^{2n}
\]
are called the {\it eigenvalues} of $\nabla$. The induced trace connection 
\[
{\rm tr}(\nabla)\colon \det(E) \rightarrow\det(E)\otimes\Omega^1_C(D)
\] 
satisfies ${\rm Res}_{t_k}({\rm tr}(\nabla))=\nu_k^+ + \nu_k^-$ and Residue Theorem yields the Fuchs relation:
\[
\deg E + \sum_{k=1}^n (\nu_k^+ + \nu_k^-)=0.
\]

\begin{remark}\label{rem:condition}
Hereafter we will fix the following data:
\begin{enumerate}
\item A $2n$-tuple of complex numbers $\nu=(\nu_1^{+},\nu_{1}^-,...,\nu_n^{+},\nu_n^-)$ satisfying the Fuchs relation
\[
1 + \sum_{k=1}^n (\nu_k^+ + \nu_k^-)=0
\]
and the generic condition: $\nu_1^{\epsilon_1}+\cdots+\nu_n^{\epsilon_n} \notin \mathbb{Z}$ for any $\epsilon_k \in \{+,- \}$, to avoid reducible connections;  and $\nu_k^+\neq\nu_k^-$ for all $k\in\{1,\dots, n\}$, so that the residues have distinguished eigenspaces;
\item A fixed trace connection $\zeta\colon \mathcal O_C(w_{\infty}) \rightarrow \mathcal O_C(w_{\infty})\otimes \Omega^1_C(D)$ satisfying 
\[
{\rm Res}_{t_k}(\zeta)=\nu_k^+ + \nu_k^-
\]
for all $k=1,...,n$.
\end{enumerate}
\end{remark}
Then we define the moduli space:
\[
\contotal = \left\{ (E,\nabla) \mathrel{\Big|} \begin{matrix} \nabla \text{ has eigenvalue } \nu,\\
\det E = \mathcal{O}_C(w_{\infty}), \, {\rm tr}\nabla = \zeta 
\end{matrix} \right\}{/\sim}
\]
where $\sim$ stands for $S$-equivalence. In fact, the condition (1) on $\nu$ implies that $\sim$ may be thought as equivalence up to isomorphism.

Algebraic constructions of moduli spaces of connections goes back to the works of Simpson and, 
in the logarithmic case, Nitsure in \cite{Nitsure}. In our setting, it is more convenient to refer to
the works of Inaba, Iwasaki and Saito \cite{IIS}, and more precisely Inaba \cite{Inaba}.
Indeed, under our generic assumption on $\nu$, 
each connection $\nabla$ on $E$ defines a unique parabolic structure, 
by selecting the eigenspace $p_k\subset E\vert_{t_k}$ associated to $\nu_k^+$ at each pole $t_k$; 
therefore, $\contotal$ can equivalently be viewed as the moduli space of parabolic connections 
as considered in the work \cite{Inaba} of Inaba.
Then it follows from \cite[Theorem 2.1, Proposition 5.2]{Inaba} that it is quasi-projective and irreducible
of dimension $2n$. Moreover, \cite[Theorem 2.2]{Inaba} shows that it is moreover smooth.
In fact, in order to fit with the stability condition \cite[Definition 2.2]{Inaba}, we set 
$\alpha_1^{(k)}=\frac{1-\mu_k}{2}$ and $\alpha_2^{(k)}=\frac{1+\mu_k}{2}$; our moduli space
therefore corresponds to the fiber $\det^{-1}(L,\zeta)$ of the determinant map considered at the beginning
of \cite[Section 5]{Inaba}. When $\nu_k^+=\nu_k^-$ for some $k$, there exist connections with scalar 
residue (apparent singular point) which give rise to a singular locus in the moduli space; 
the role of the parabolic structure in \cite{Inaba} is to get a smooth moduli space even in that case.

Then the moduli space $\contotal$ is a smooth irreducible quasi-projective variety of dimension $2n$, provided that the condition (1) on $\nu$ is satisfied. The case $n=1$ is not covered by \cite{Inaba}, but it follows from \cite{Lame}; we will discuss this case at the end of this section.

 There exists also an analytical construction for $\contotal$ following Nakajima \cite{NaHK}. There, moduli spaces of rank two parabolic connections were constructed via hyper-kähler quotients. Thus $\contotal$ has a holomorphic symplectic structure $\omega$.

In the study of $\contotal$, it is useful to consider the quasi-parabolic bundles underlying a connection. 
Let us assume $\nu_k^+\neq\nu_k^-$ for all $k\in\{1,\dots, n\}$. Given a connection $(E,\nabla)$, we associate, for each $k = 1, \dots, n$, a pair of ``positive'' and ``negative'' eigenspaces of ${\rm Res}_{t_k}(\nabla)$
\[
p_{k}^{+}(\nabla), p_{k}^{-}(\nabla)\in \mathbb P (E_{t_k})
\]
corresponding to the eigenvalues $\nu_{k}^{+}$ and $\nu_{k}^{-}$ respectively. 

Given an $n$-tuple $\epsilon = (\epsilon_1,\cdots, \epsilon_n)$, where each $\epsilon_i\in\{+,-\}$, we denote
\[
{\bf p}^{\epsilon}(\nabla) = \{p_1^{\epsilon_1}(\nabla),\cdots, p_n^{\epsilon_n}(\nabla)\}
\]
and consider $(E,{\bf p}^{\epsilon}(\nabla))$ the quasi-parabolic vector bundle defined by these directions. 

\begin{remark}\label{rem:genindec}
The hypothesis that $\nu_1^{a_1}+\cdots+\nu_n^{a_n} \notin \mathbb Z$, for every $a\in \{+.-\}^n$, ensures that $(E,{\bf p}^{\epsilon}(\nabla))$ is an indecomposable quasi-parabolic bundle. Indeed, if $(L, {\bf p})$ is rank one direct summand of $(E,{\bf p}^{\epsilon}(\nabla))$ then the residues of the induced connection on $L$ are either $\nu_j^+$ or $\nu_j^-$; their sum is $-\deg L$, see \cite[Corollary 2.3]{FL}.
\end{remark}

One can then ask for the stability of these quasi-parabolic bundles with respect to some weight. We define
\[
\contotal_{st} = \left\{ (E,\nabla) \in \contotal \mathrel{\Big|} \begin{matrix}
\exists\, \epsilon\in\{+,-\}^n \text{ and } \exists\, I \subset \{1, \dots, n \} \\ \text{such that } (E,{\bf p}^{\epsilon}(\nabla)) \in {\rm Bun}^I
\end{matrix} \right\}.
\]
Recall that $|I|$ is always assumed to be even. 
It follows that 
\[
\contotal = \contotal_{st} \sqcup \ZZ_n
\]
where $\ZZ_n$ denotes the complement of $\contotal_{st}$. Our aim in the next subsections is to describe these varieties. We will show that $\contotal_{st}$ can be covered by simple open subsets and that $\ZZ_n$ falls in two cases: either $n$ is even and $\ZZ_n= \emptyset$ or $n$ is odd and $\ZZ_n$ has four connected components, each one is a quotient of $\mathbb{C}^{n+1}$ by a free affine action of the additive group $(\mathbb{C},+)$, hence isomorphic to $\mathbb{C}^n$.

\subsection{Connections on \texorpdfstring{$E_1$}{E1}} Our main building block in the description of $\contotal_{st}$ is the space defined by 
\[
{\rm Con}^\nu = \left\{(E,\nabla) \in \contotal \mid E = E_1 \right\}.
\]
Note that every underlying quasi-parabolic bundle lies in ${\rm Bun}={\rm Bun}^{\emptyset}$, see Proposition \ref{prop:1stchamb} and Definition \ref{def:bunI}. The same proposition shows that ${\rm Bun} \simeq (\mathbb{P}^1)^n$. We will see that ${\rm Con}^\nu$ has a similar description. 

Let $\Delta\subset \mathbb P^1\times \mathbb P^1$ be the diagonal and let $S:=(\mathbb P^1\times \mathbb P^1)\backslash\Delta$ be its complement. Then we define a map 
\[
\begin{tikzcd}[row sep=1.pt, ampersand replacement=\&]
{\rm Par}\colon {\rm Con}^\nu \arrow[r] \& S^n\\
	\quad\quad (E_1, \nabla) \arrow[r, maps to] \& (p_1^+(\nabla),p_1^-(\nabla) ; \cdots ; p_n^+(\nabla),p_n^-(\nabla) )
\end{tikzcd}.
\]
This map is in fact an isomorphism.

\begin{thm}\label{thm:parconSn}
	The map	$\Par \colon {\rm Con}^{\nu}\rightarrow S^n$ is an isomorphism. 
\end{thm}

\begin{proof}
We may factor the map $\Par$ into two parts:
\[
{\rm Con}^{\nu} \xrightarrow{{\rm Res}_D}  \bigoplus_{j=1}^n \left\{ A \in {\rm End}(E_1)_{t_j} \mid A\textrm{ has eigenvalues } \nu_j^+, \nu_j^- \right\} \stackrel{\pi}{\longrightarrow}  S^n 
\]
where ${\rm Res}_D$ is the residue map and $\pi$ sends an endomorphism to the ordered pair of eigenspaces, $+$ then $-$. Both maps are indeed morphisms and we will see that they are both isomorphisms.

First we deal with $\pi$. Given an endomorphism $A$ of $(E_1)_{t_j}$ with eigenvalues $\nu_j^+ \neq \nu_j^-$ we can associate the respective eigenspaces $(p^+, p^-)$. Conversely, fix a local frame for $E_1$ around $t_j$ giving coordinates $(E_1)_{t_j} \simeq \mathbb{C}^2$. Then to $((z:w), (u:v)) \in S$ we can associate the matrix
\[
A = \frac{1}{zv-uw}\begin{pmatrix}
	z & u \\ w & v
\end{pmatrix} \begin{pmatrix}
	\nu_j^+ & 0 \\ 0 & \nu_j^-
\end{pmatrix}\begin{pmatrix}
	v & -u \\ -w & z
\end{pmatrix}
\]
defining an element of ${\rm End}(E_1)_{t_j}$. Thus $\pi$ is bijective. 

Next we deal with ${\rm Res}_D$. If $\nabla_1, \nabla_2 \in {\rm Con}^{\nu}$ have the same residues at $t_j$, for $j = 1, \dots, n$, then $\nabla_1 - \nabla_2 \in H^0(\mathcal{E}nd(E_1)\otimes \Omega_C)$. Since $E_1$ is simple and $\Omega_C = \mathcal{O}_C$ we have $\nabla_1 - \nabla_2 = id_E \otimes \eta$ for some holomorphic $1$-form $\eta$. On the other hand, the definition of ${\rm Con}^{\nu}$ imposes that $\nabla_1$ and $\nabla_2$ have the same trace. Thus $\eta = 0$ and $\nabla_1 = \nabla_2$, proving that ${\rm Res}_D$ is injective.  

It remains to prove that ${\rm Res}_D$ is surjective. Fix $A_j \in {\rm End}(E_1)_{t_j}$ for $j = 1, \dots n$ such that $A_j$ has eigenvalues $\nu_j^+, \nu_j^-$. Since $E_1$ is simple and 
\[
\deg(E_1) + \sum_{j=1}^n {\rm tr}(A_j) = 1 + \sum_{j=1}^n \nu_j^+ + \nu_j^- = 0,
\]
\cite[Lemma 3.2]{BISWAS} guarantees the existence of a connection $\nabla_0$ with residue $A_j$ at $t_j$. Moreover, ${\rm tr}(\nabla_0) - \zeta = \theta \in  H^0(\Omega_C)$ a global holomorphic $1$-form. Hence $\nabla = \nabla_0 - \frac{1}{2}id_{E_1} \theta$ has the prescribed residues and ${\rm tr}(\nabla) = \zeta$ so that $\nabla\in {\rm Con}^{\nu}$.
\end{proof}

We will give an alternative proof of this theorem using explicit computations of Fuchsian systems, see Remark \ref{rem:consys}.


\begin{cor}
${\rm Con}^{\nu}$ is an affine variety. 
\end{cor}

\begin{proof}
Since the diagonal $\Delta\subset \mathbb{P}^1\times \mathbb{P}^1$ supports an ample divisor, its complement, $S$, is affine. Therefore ${\rm Con}^{\nu} \simeq S^n$ is also affine.
\end{proof}

In the next subsection we will see that $S^n$ is a local model for $\contotal_{st}$.

\subsection{Description of \texorpdfstring{$\contotal_{st}$}{Con st} }
From the definition, $\contotal_{st}$ is the space of (isomorphism classes of) connections $(E,\nabla)$ for which there exist $I\subset \{1, \dots, n\}$, with $|I|$ even, and $\epsilon \in \{-,+\}^n$ such that $(E, {\bf p}^\epsilon(\nabla)) \in {\rm Bun}^I$, i.e. $(E, {\bf p}^\epsilon(\nabla))$ is $\mu$-stable for any $\mu\in \mathfrak{C}_I$. For each $I$ and $\epsilon$ we define
\[
{\rm Con}_{I,\epsilon}^{\nu} := \left\{ (E,\nabla)\in \contotal \mid (E, {\bf p}^\epsilon(\nabla)) \in {\rm Bun}^I \right\},
\]
hence we get a decomposition
\begin{eqnarray}\label{coverforcon}
\contotal_{st} = \bigcup\limits_{I,\epsilon} {\rm Con}_{I,\epsilon}^{\nu}.
\end{eqnarray}
Note that ${\rm Con}_{\emptyset,\epsilon}^{\nu} = {\rm Con}^{\nu}$ for any $\epsilon$. Next we will see that, for generic $\nu$, each ${\rm Con}_{I,\epsilon}^{\nu}$ is isomorphic to $S^n$. More precisely, we will prove that ${\rm Con}_{I,\epsilon}^{\nu}$ is isomorphic to ${\rm Con}^{\lambda}$, for some eigenvalue $\lambda$ to be determined. Consider 
\[
\pi_\epsilon\colon {\rm Con}_{I,\epsilon}^{\nu} \longrightarrow {\rm Bun}^I
\]
the forgetful morphism.

\begin{prop}\label{prop:fiberpres}
The map $\elm_I$ induces a fiber-preserving isomorphism $\Phi_I^\epsilon$:
\[
\begin{tikzcd}[column sep = 33.pt]
	{\rm Con}_{I,\epsilon}^{\nu} \arrow[d, "\pi_\epsilon"'] \arrow[r, "\Phi_{I}^\epsilon"] & {\rm Con}^{\lambda}\arrow[d, "\pi_\epsilon"] \\
	{\rm Bun}^I\arrow[r, "\elm_{I}"] & {\rm Bun}
\end{tikzcd}
\]
for $\lambda=(\lambda_{1}^{+}, \lambda_{1}^{-}, \cdots, \lambda_{n}^{+}, \lambda_{n}^{-})$ defined as follows. If $k \not\in I$ then $\lambda_{k}^+ = \nu_{k}^+$ and $\lambda_{k}^- = \nu_{k}^-$, and if $k\in I$ then
\[
\lambda_{k}^{\epsilon_k} = \nu_{k}^{-\epsilon_k}+\frac{1}{2},  \text{ and }  \lambda_{k}^{-\epsilon_k} = \nu_{k}^{\epsilon_k}- \frac{1}{2},
\]
where $\{\epsilon_k,-\epsilon_k \}=\{+,-\}$.
\end{prop}

\begin{proof}
Given $(E,\nabla)\in {\rm Con}_{I,\epsilon}^{\nu}$ we will perform an elementary transformation centered in ${\bf p}^{\epsilon}(\nabla)$. Recall that $\elm_I$ sends $(E,{\bf p}^{\epsilon}(\nabla))$ to $\Big(E'\otimes L, {\bf p'}\Big)$ where $E'$ is obtained from the exact sequence 
\[
0\longrightarrow E' \stackrel{\alpha}{\longrightarrow} E \longrightarrow \bigoplus_{i\in I}\left(E_{t_i}/p_i^{\epsilon_i}(\nabla)\right) \longrightarrow 0,
\] 
and $L$ is a square root $\mathcal O_C\left(\sum_{i\in I}t_i\right)$. The pullback connection $\alpha^*(\nabla)$ has the following property, which can be verified in local coordinates. For $k\not\in I$, the eigenvalue at $t_k$ are the same $\{ \nu_k^+, \nu_k^-\}$. For $k \in I$, the eigenvalues are 
\[
(\nu_{k}^{-\epsilon_k}+1, \nu_{k}^{\epsilon_k})
\] 
and $\ker( \alpha_{t_k})$ corresponds to $\nu_{k}^{-\epsilon_k}+1$.

Now let $\xi\colon L\rightarrow L\otimes\Omega_C^1(D)$ be the rank one connection defined as follows. Let $\{U_i\}$ be a trivializing cover for $E,E'$ and $L$ and let $\{G_{ij}\}, \{G_{ij}'\}$ and $\{h_{ij}\}$ be the respective cocycles. Assume further that $\det G_{ij} = h_{ij}^2\det G_{ij}'$ for every pair $(i,j)$. Then 
\[
\xi|_{U_i} = d - \frac{1}{2}{\rm tr}\left(\alpha_i^{-1}d\alpha_i\right),
\]
where $\alpha_i$ is the local expression for $\alpha$. Note that ${\rm tr}(\alpha^*(\nabla) \otimes \xi) = {\rm tr}(\nabla)$ and 
\[
{\rm Res}_{t_k}\xi = \begin{cases}
	-\frac{1}{2}, & k\in I;\\
	0, & k \not\in I.
\end{cases}
\]

The map $\Phi_I^\epsilon$ is then defined as 
\[
\Phi_I^\epsilon(E, \nabla) = (E'\otimes L, \alpha^*(\nabla)\otimes \xi).
\]
Since it can be reversed by the same process, we have the isomorphism. Moreover, the diagram in the statement commutes from the construction of $\Phi_I^\epsilon$.
\end{proof}

To give an isomorphism ${\rm Con}_{I,\epsilon}^{\nu}\simeq S^n$ we just need to require that $\lambda$ satisfies the hypothesis of Theorem \ref{thm:parconSn}, i.e. $\lambda_{k}^+ \neq \lambda_{k}^-$ and in Remark \ref{rem:condition}.

\begin{cor}
If, for every $k= 1, \dots, n$, $\nu_{k}^{+}-\nu_{k}^{-}\notin \{0, 1, -1\}$ then, for every $I\subset \{1,\dots, n\}$, with $|I|$ even, and every $\epsilon \in \{+,-\}^n$,
\[
{\rm Con}_{I,\epsilon}^{\nu} \simeq S^n.
\]
\end{cor}

\begin{proof}
For $k\not\in I$ we have $\lambda_{k}^+ - \lambda^- = \nu_{k}^+ - \nu_{k}^-$ and for $k\in I$ we have $\lambda_{k}^+ - \lambda^- = \nu_{k}^- - \nu_{k}^+ \pm 1$. The result follows from Proposition \ref{prop:fiberpres} and Theorem \ref{thm:parconSn} once we have $\nu_{k}^{+}-\nu_{k}^{-}\notin \{0, 1, -1\}$.
\end{proof}

Hereafter we will also assume that $\nu_{k}^{+}-\nu_{k}^{-}\notin \{0, 1, -1\}$ for every $k = 1, \dots, n$.
We will also denote by $\pi_\epsilon$ the projection $\pi_\epsilon\colon S^n \rightarrow (\mathbb{P}^1)^n$ that makes the following diagram commute
\[
\begin{tikzcd}[column sep = 33.pt]
	{\rm Con}_{I,\epsilon}^{\nu} \arrow[d, "\pi_\epsilon"'] \arrow[r] & S^n \arrow[d, "\pi_\epsilon"] \\
	{\rm Bun}^I\arrow[r] & (\mathbb{P}^1)^n
\end{tikzcd}
\]
and we define $\widetilde{\Gamma}_{I,\epsilon} = \pi_\epsilon^{-1}( \Gamma_I)$, where $\Gamma_I$ is the hypersurface from Definition \ref{def:gammaI}.

\begin{prop}\label{prop:complement}
	Let $I, J\subset \{1,\dots, n\}$ with even cardinalities and fix $\epsilon \in\{+,-\}^n$. Then 
	\[
	{\rm Con}^{\nu}_{I,\epsilon} \setminus {\rm Con}_{J,\epsilon}^{\nu} \simeq {\rm Con}^{\nu}_{J,\epsilon} \setminus {\rm Con}_{I,\epsilon}^{\nu} \simeq \widetilde{\Gamma}_{I\Delta J,\epsilon}
	\]
	where $I\Delta J = (I\cup J) \setminus (I\cap J)$ is their symmetric difference.
\end{prop}

\begin{proof}Note that $|I\Delta J|$ is even and $\elm_I \circ \elm_J = \elm_{I\Delta J}$; hence $\Phi_I^\epsilon$ gives an isomorphism 
\[
{\rm Con}^{\nu}_{I,\epsilon} \setminus {\rm Con}_{J,\epsilon}^{\nu} \simeq {\rm Con}^{\lambda} \setminus {\rm Con}_{I\Delta J,\epsilon}^{\lambda}.
\]
An element of ${\rm Con}^{\lambda} \setminus {\rm Con}_{I\Delta J,\epsilon}^{\lambda}$ is a connection whose underlying quasi-parabolic bundle $(E_{1}, \textbf{p}^{\epsilon}({\nabla}))$ is stable for the weights $\mu\in\mathfrak{C}$ but its image under $\elm_{I\Delta J}$ is unstable or, equivalently, it is $\mu$-stable but $\varphi_{I\Delta J}(\mu)$-unstable. Therefore, using Theorem \ref{thm:parconSn} and Remark \ref{rem:elmGammaI} we get
\[
{\rm Par}({\rm Con}^{\lambda}\setminus {\rm Con}_{I\Delta J,\epsilon}^{\lambda}) 	 = \widetilde{\Gamma}_{I\Delta J,\epsilon}.
\]
\end{proof}

\begin{thm}\label{thm:glue}
Assume and $\nu_1^{a_1}+\cdots +\nu_n^{a_n} \notin \mathbb Z$, for any $a_{k}\in \{+,-\}$, and that $\nu_{k}^{+}-\nu_{k}^{-}\notin \{0, 1, -1\}$ for $k\in\{1, \cdots, n\}$. Then $\contotal_{st}$ is obtained by gluing a finite number of copies of $S^n$ via birational maps
\[
\Psi_{J,I}^{\delta,\epsilon}\colon S^n \dashrightarrow S^n.
\]
depending on $\epsilon, \delta \in \{+,-\}^n$ and $I, J \subset \{1, \cdots, n\}$.
Moreover, if $\epsilon = \delta$, then $\Psi_{J,I}^{\delta,\epsilon}$ preserves the fibers of $\pi_\epsilon$. 
\end{thm}

\begin{proof} Since $\contotal_{st} = \cup_{I,\epsilon} {\rm Con}_{I,\epsilon}^{\nu}$, \eqref{coverforcon}, we may give local charts ${\rm Par } \circ \Phi_I^\epsilon \colon {\rm Con}^\nu_{I,\epsilon} \rightarrow S^n$. Note that ${\rm Con}^\nu_{I,\epsilon}\cap {\rm Con}^\nu_{J,\delta}$ is Zariski open. Then the maps $\Psi_{J,I}^{\delta,\epsilon}$ are defined by extending the transition maps. They fit in the following diagram.
\[
\begin{tikzcd}[column sep = 23pt]
S^n\arrow[d, "\pi_\epsilon"]\arrow[rrrrrr, bend left = 20, dashed, start anchor ={[yshift = 8pt]}, end anchor = {[yshift = 8pt]}, "\Psi_{J,I}^{\delta,\epsilon}"'] 
& {\rm Con}^\lambda \arrow[l, "{\rm Par}"']\arrow[d, "\pi_\epsilon"']
& {\rm Con}^\nu_{I,\epsilon}\arrow[d, "\pi_\epsilon"'] \arrow[l, "\Phi_I^\epsilon"'] 
& {\rm Con}^\nu_{I,\epsilon}\cap {\rm Con}^\nu_{J,\delta} \arrow[l,symbol = \subset]\arrow[r,symbol = \subset]
& {\rm Con}^\nu_{J,\delta}\arrow[r, "\Phi_J^\delta"]\arrow[d, "\pi_\delta"] 
& {\rm Con}^\rho\arrow[d,"\pi_\delta"] \arrow[r,"{\rm Par}"] 
& S^n\arrow[d,"\pi_\delta"] \\ 
(\mathbb{P}^1)^n 
& {\rm Bun}\arrow[l] 
& {\rm Bun}^I \arrow[l]
&& {\rm Bun}^J\arrow[r] 
& {\rm Bun}\arrow[r]
& (\mathbb{P}^1)^n
\end{tikzcd}
\]

If $\delta = \epsilon$ then we can complete the diagram. By abuse of notation, we may write
\[
{\rm Bun}^{IJ} = {\rm Bun}^I \cap {\rm Bun}^J
\]
the space parabolic bundles in ${\rm Bun}^I$ that are also $\mu$-stable for $\mu \in \mathfrak{C}_J$ and vice-versa. Thus we get a birational map ${\rm Bun}^I \dashrightarrow {\rm Bun}^J$ extending the identity on ${\rm Bun}^{IJ}$. Hence we get
\[
\begin{tikzcd}
	S^n \arrow[d, "\pi_\epsilon"']\arrow[r, dashed, "\Psi_{J,I}^{\epsilon,\epsilon}"]
	& S^n \arrow[d, "\pi_\epsilon"]\\
	(\mathbb{P}^1)^n \arrow[r, dashed]
	& 	(\mathbb{P}^1)^n 
\end{tikzcd} 
\] 
\end{proof}

\subsection{Connections with unstable parabolic bundles.}
Now we describe the space $\ZZ_n$ of (isomorphism classes of) connections such that every underlying parabolic bundle is not semistable. We will show that it falls in two cases:
\begin{enumerate}
	\item If $n$ is even then $\ZZ_n= \emptyset$;
	\item If $n$ is odd then $\ZZ_n$ has four connected components, each isomorphic to $\mathbb{C}^{n}$.
\end{enumerate}
Note that in the last case $\dim \ZZ_{n} = n$. Let us make our first reduction.

\begin{lemma}\label{lem:redsigma}
If $n$ is even then $\ZZ_n$ is empty and if $n$ is odd then
\begin{equation}\label{sigma}
	\ZZ_{n} = \{ (E,\nabla)\in \contotal \mid E = L \oplus L^{-1}(w_\infty) \text{ with } L^2 = \mathcal{O}_C(D + w_\infty)\}.
\end{equation}
\end{lemma}

\begin{proof}
Suppose that $\ZZ_n \neq \emptyset$ and let $(E,\nabla)\in \ZZ_n$. We will show that $E = L \oplus L^{-1}(w_\infty)$ with $L^2 = \mathcal{O}_C(D + w_\infty)$. In particular, $n$ must be odd.

We know that $(E, {\bf p}^\epsilon(\nabla))$ is indecomposable for any $\epsilon$, see Remark \ref{rem:genindec}. Then we may apply Lemma \ref{lem:indec}. We cannot have $E = E_1$ since it would imply $(E_1,\nabla)\in {\rm Con}^\nu$; hence $E = L \oplus L^{-1}(w_\infty)$. In this case, we may take $ \epsilon $ such that $p_k^{\epsilon_k}(\nabla)\not\subset L_{t_k}$ for every $k=1, \dots , n$. Note that, by Lemma \ref{lem:stparbun}, the case $2\deg L \leq n$ is not possible. Therefore Lemma \ref{lem:indec} implies that $L^2 = \mathcal{O}_C(D + w_\infty)$. 

Conversely, any connection on $L \oplus L^{-1}(w_\infty)$, with $L^2 = \mathcal{O}_C(D + w_\infty)$, represents a point in $\ZZ_n$, see Remark \ref{rem:unstparbun}.
\end{proof}

Next we will describe the connections in $\ZZ_{n}$. In order to do so, we compute the logarithmic Atiyah class $\phi^A_E \in {\rm End}(E)^\vee$ whose vanishing establishes the existence of a connection with prescribed residues, see \cite{BISWAS}. Let $T \in{\rm End}(E)$ then $\phi^A_E$ is defined by
\[
\phi^A_E(T) = \phi^0_E(T) + \sum_{j=1}^n {\rm tr}\left(A_j T(t_j)\right)
\]
where $A_j$ is the residue endomorphism over $t_j$ and $\phi^0_E$ is the classical Atiyah class, see \cite{At}. In our case,
\[
A_j = \begin{pmatrix}
	u_j & a_j \\ v_j & b_j
\end{pmatrix} \begin{pmatrix}
	\nu_j^+ & 0 \\ 0 & \nu_j^-
\end{pmatrix}\begin{pmatrix}
	b_j & -a_j \\ -v_j & u_j
\end{pmatrix}
\]
where $u_jb_j -a_jv_j = 1$. Here we may take local coordinates around each $t_j$ such that $L$ and $L^{-1}(w_\infty)$ correspond to $(1:0)$ and $(0:1)$, respectively.

Note that any direction $p_k^{\epsilon_k}(\nabla)$ lies outside $L$, otherwise there would exist a choice of parabolic directions ${\bf p}^{\epsilon}(\nabla)$ such that $(E,{\bf p}^{\epsilon}(\nabla))$ is decomposable and this would force a relation on eigenvalues $\nu$. Indeed, we can find an embedding of $L^{-1}(w_{\infty})$ passing through $n-1$ directions away from $L$. Then we suppose, without loss of generality, that our directions are as in the Figure \ref{fig:confsigma}. In particular $(u_j, v_j)= (0,1)$ and $a_j =-1$ for $j\geq 2$, and $u_1v_1 \neq 0$. Up to applying a diagonal automorphism of $E$ we suppose that $u_1 = v_1 = 1$, i.e. $p_1^+(\nabla) =(1:1)$.

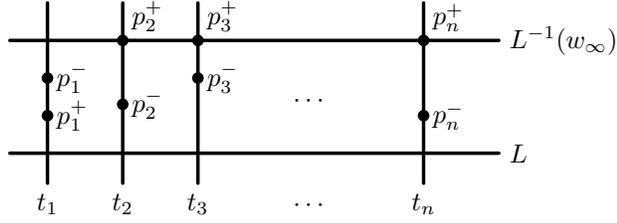
\begin{figure}[t]
	\centering
	\begin{tikzpicture}[line cap=round,line join=round,>=triangle 45,x=1.0cm,y=1.0cm]
		\begin{axis}[
			x=1.0cm,y=0.5cm,
			axis lines= left,
			axis line style={draw=none},
			ymajorgrids=false,
			xmajorgrids=false,
			xmin=2.8,
			xmax=11.8,
			ymin=0.,
			ymax=6.5,
			xtick= \empty,
			ytick=\empty,] 
			\draw [line width=1.3 pt] (4.,6.)-- (4.,1.2) node[anchor=north] {$t_1$};
			\draw [line width=1.3 pt] (5.,6.)-- (5.,1.2) node[anchor=north] {$t_2$};
			\draw [line width=1.3 pt] (6.,6.)-- (6.,1.2) node[anchor=north] {$t_3$};
			\draw [line width=1.3 pt] (9.,6.)-- (9.,1.2) node[anchor=north] {$t_n$};
			\draw [line width=1.33 pt] (3.5,5.)-- (10.,5.) node[anchor= west] {$L^{-1}(w_\infty)$} ;
			\draw [line width=1.33 pt] (3.5,2.)-- (10.,2.) node[anchor=west] {$L$};
			\fill[color=black] ( 5.,5.) circle (2.2pt)node[anchor=south west] { $p_2^+$};
			\fill[color=black] ( 6.,5.) circle (2.2pt)node[anchor=south west] { $p_3^+$};
			\fill[color=black] ( 9.,5.) circle (2.2pt)node[anchor=south west] { $p_n^+$};
			\fill[color=black] ( 5.,3.3) circle (2.2pt)node[anchor= west] { $p_2^-$};
			\fill[color=black] ( 6.,4.) circle (2.2pt)node[anchor= west] { $p_3^-$};
			\fill[color=black] (9.,3.) circle (2.2pt)node[anchor= west] { $p_n^-$};
			\fill[color=black] ( 4.,4.) circle (2.2pt)node[anchor= west] { $p_1^-$};
			\fill[color=black] ( 4.,3.) circle (2.2pt)node[anchor= west] { $p_1^+$};
			\draw (7.5,0.6) node{ $\cdots$};
			\draw (7.5,3.4) node{ $\cdots$};
		\end{axis}
	\end{tikzpicture}
	\caption{Possible configuration of directions for $(E, \nabla) \in \ZZ_n$.}
	\label{fig:confsigma}
\end{figure}

Note that ${\rm End}(E)$ is generated (as a vector space) by the identity, nilpotent endomorphisms and the projection to $L$. For the identity, $\phi^A_E(1_E)$ gives the Fuchs relation that we already know is valid. Let $\beta \in {\rm H^0}(C, L^2(-w_\infty))$ and define
\[
P(\beta) := \phi^A_E\left(\begin{pmatrix}
	0 & \beta \\ 0 & 0
\end{pmatrix} \right) = b_1\beta(t_1)(\nu_1^+-\nu_1^-) + \sum_{j \geq 2}b_j\beta(t_j)(\nu_j^+-\nu_j^-).
\]
For $j\geq 2$ let $\beta_j \in {\rm H^0}(C, L^2(-w_\infty))$ with the following property: $\beta_j(t_k) = 0$ if $k \neq 1,j$ and $\beta_j(t_j) = 1$. These sections are unique. In particular, $\beta_j(t_1) \neq 0$ and we have
\[
P(\beta_j) = b_1\beta_j(t_1)(\nu_1^+-\nu_1^-) + b_j(\nu_j^+-\nu_j^-)
\] 
Note that the image of evaluation map ${\rm H^0}(C, L^2(-w_\infty)) \rightarrow \mathbb{C}^n$, $\beta \mapsto (\beta(t_1), \dots, \beta(t_n))$, has dimension $(n-1)$; hence the images of the $\beta_j$ define basis. Therefore $P(\beta)=0$ for every $\beta \in {\rm H^0}(C, L^2(-w_\infty))$ if and only if $P(\beta_j) = 0$ for $j\geq 2$, i.e. the $\nu_j^-$ direction is 
\[
(-1 : b_j) =\left(\nu_j^+ -\nu_j^- : b_1\beta_j(t_1)(\nu_1^+-\nu_1^-) \right) 
\]

For the projection to $L$ we have 
\[
\phi^A_E\left(\begin{pmatrix}
	1 & 0 \\ 0 & 0
\end{pmatrix} \right) = \deg L + \sum_{j=1}^n u_jb_j\nu_j^+ - a_jv_j\nu_j^- = b_1(\nu_1^+-\nu_1^-) + \frac{n+1}{2} + \sum_{j=1}^n \nu_j^- =0.
\]
This implies that the directions over $t_1$ are $p_1^+(\nabla) =(1:1)$ and
\[
p_1^-(\nabla)= \left(b_1 -1 : b_1 \right) = \left( \frac{n+1}{2} + \nu_1^+ + \sum_{ j\geq 2}^n \nu_j^- : \frac{n+1}{2} + \sum_{j=1}^n \nu_j^-\right)
\]
Therefore the residues are completely independent of the isomorphism class of $(E,\nabla)$, i.e. the residues of every connection in $\ZZ_n$ are, up to ${\rm Aut}(E)$, in the above configuration. Also note that any two connections with these residues differ by an element of ${\rm Hom}(E, E\otimes \Omega_C)$ with vanishing trace. From this discussion we can prove the following result.

\begin{thm}\label{thm:dessigma}
	Let $n$ be an odd integer. Then $\ZZ_n$ has four connected components, each of them being isomorphic to $\mathbb{C}^n$. 
\end{thm} 

\begin{proof}
	First note that there exist precisely four possibilities for the underlying vector bundle of a connection in $\ZZ_n$. Indeed, Lemma \ref{lem:redsigma} shows that any such vector bundle is $E= L\oplus L^{-1}(w_\infty)$ where $L$ is such that $L^2 = \mathcal{O}_C(D+w_\infty)$. Twisting by $2$-torsion line bundles yields four non-isomorphic possibilities for $L$. Hence four non-isomorphic possibilities for $E$. Therefore $\ZZ_n$ has four connected components.
	
	Fix one such $E$ and denote $\ZZ_n^E$ the corresponding component of $\ZZ_n$. Up to the action of ${\rm Aut}(E)$, we can fix a configuration of directions as in Figure \ref{fig:confsigma} so that we may only consider connections on $E$ that have this configuration. Note that since $L^2 = \mathcal{O}_C(D+w_\infty)$ the stabilizer of such configuration in ${\rm Aut}(E)$ is a copy of the additive group $(\mathbb{C},+)$ generated by
	\[
	\begin{pmatrix}
		1 & \beta \\ 0 & 1
	\end{pmatrix}
	\]
	where $\beta \in H^0(L^2(-w_\infty))\setminus \{0\}$ vanishes on $D$. On the other hand, if we fix a connection $\nabla_0$ on $E$, for any other connection $\nabla$, the difference $\nabla- \nabla_0 \in {\rm Hom}(E, E\otimes \Omega_C)$ is a holomorphic Higgs field. Since $\nabla$ and $\nabla_0$ must have the same trace, this Higgs field is traceless. Thus we have an isomorphism
	\[
	\begin{tikzcd}[row sep = 2pt]
		\ZZ_n^E \arrow[r] & {\rm Higgs}_0(E)/(\mathbb{C},+)\\
		\left[\nabla\right] \arrow[r,maps to] & \left[\nabla - \nabla_0\right]
	\end{tikzcd}
	\]
	where ${\rm Higgs}_0(E)$ is the space of traceless Higgs fields. Note that ${\rm Higgs}_0(E) \simeq \mathbb{C}^{n+1}$. 
	
	We now explicitly describe the action of $(\mathbb{C},+)$ on ${\rm Higgs}_0(E)$. Locally, we can write 
	\[
	\nabla_0 = d+\begin{pmatrix}
		a_0 & b_0 \\ c_0 & d_0
	\end{pmatrix} \text{ and } \varphi = \begin{pmatrix}
		a_1 & b_1 \\ 0 & -a_1
	\end{pmatrix}.
	\]
	Then the action of $t\in \mathbb{C}$ is given by
	\begin{align*}
		t\cdot\varphi &= \begin{pmatrix}	1 & -t\beta \\ 0 & 1 \end{pmatrix} \left[
		\begin{pmatrix}	0 & td\beta \\ 0 & 0 \end{pmatrix} + 
		\begin{pmatrix}a_0+a_1 & b_0+b_1 \\ c_0& d_0-a_1 \end{pmatrix}
		\begin{pmatrix}	1 & t\beta \\ 0 & 1 \end{pmatrix}\right] - 
		\begin{pmatrix}a_0 & b_0 \\ c_0 & d_0 \end{pmatrix} = \\
		& = \begin{pmatrix} 1 & -2t\beta \\ 0 & 1	\end{pmatrix} \begin{pmatrix}a_1 & b_1 \\ 0 &-a_1 \end{pmatrix} + 
		\begin{pmatrix} -tc_0\beta & t[(a_0-d_0)\beta +d\beta] - t^2c_0\beta^2 \\ 0 & tc_0\beta \end{pmatrix}.
	\end{align*}
	In particular, the action is given by affine transformations. Also note that $c_0 = 0$ would force an integer relation $\nu_1^{\epsilon_1}+\cdots+\nu_n^{\epsilon_n} \in \mathbb{Z}$ which is not possible; hence the action is free. This concludes the proof since the quotient of an affine free action of $(\mathbb{C},+)$ on $\mathbb{C}^{n+1}$ must be isomorphic to $\mathbb{C}^n$, see \cite[Corollary 7]{PutFreeAct}.
\end{proof}

In addition, notice that given $(E,\nabla)$ representing a point in $\ZZ_n$ (as in \eqref{sigma}) we may perform an elementary transformation centered at all parabolic directions to get $(E',\nabla')$. The subbundle $L\hookrightarrow E = L\oplus L^{-1}(w_\infty)$ becomes $U = L(-D)\otimes M \hookrightarrow E'$ where $M^2 = \mathcal{O}_C(D-w_\infty)$.  From $L^2 = \mathcal{O}_C(D+w_\infty)$ we have that $U^2 = \mathcal{O}_C$ and we have an extension
\begin{equation}\label{seq:extE0}
    0\longrightarrow U \longrightarrow E' \longrightarrow U \longrightarrow 0.
\end{equation}
Since $L$ does not pass through any parabolic direction on $E$, its transformed $U$ passes through every parabolic direction on $E'$. Thus $E'$ is indecomposable, otherwise we would have a decomposable quasi-parabolic bundle contradicting Remark \ref{rem:genindec}.  Therefore $U$ is one of the $4$ torsion line bundles and $E'$ is the corresponding unique indecomposable extension.


 We conclude that any connection representing a point in $\ZZ_n$ can be obtained from a connection over an indecomposable vector bundle $E'$ as in \eqref{seq:extE0}, by performing an elementary transformation centered in $n$ parabolic directions which lie in the unique maximal subbundle. Note that there exist $n$ directions to choose and $\dim {\rm Higgs}_0(E') = 1$ giving dimension $n+1$ for the space of connections on $E'$. Taking the quotient by the action of the automorphism group of the corresponding vector bundle, we get dimension $n$.

\subsection{The classical Painlev\'e case}
Here it is worth to compare the case $n=1$ with the Painlev\'e case, i.e. those logarithmic connections over the $4$-punctured Riemann sphere. We assume that $D=w_{\infty}$ and let $(\nu,-\nu-1)$, $\nu\in \mathbb C\setminus \mathbb Z$, denote the eigenvalues. The moduli space $\contotal$ parametrizes $(E,\nabla)$, where $\nabla$ is a logarithmic connection over $C$ with a single logarithmic pole over the point $w_{\infty}$, i.e. it is a Lam\'e connection. The whole moduli space writes 
\[
\contotal={\rm Con}^\nu\sqcup \ZZ_{1}
\]
where ${\rm Con}^\nu\simeq S$ is formed by $(E,\nabla)$ with $E=E_1$. It follows from Theorem \ref{thm:dessigma} that $\ZZ_{1}$ is a union of four irreducible curves isomorphic to $\mathbb C$:
\begin{equation}
	\ZZ_{1}^i = \{ (E,\nabla)\in \contotal \mid E = L \oplus L^{-1}(w_\infty) \text{ with } L = \mathcal{O}_C(w_i)\}
\end{equation}
for $i\in\{0,1,\lambda, \infty\}$, where $w_i$ are Weierstrass points. 

We let ${\rm Con}^{\theta}(\mathbb P^1,T)$, $T=0+1+\lambda+\infty$, denote the moduli space of logarithmic connections over the $4$-punctured Riemann sphere $(\mathbb P^1, T)$, with eigenvalue 
\[
\theta = \left ( \pm\frac{1}{4}, \pm\frac{1}{4}, \pm\frac{1}{4}, \theta^{\pm} \right )\;;\;\; \theta^+=\frac{\nu}{2}-\frac{1}{4}\;,\;\;\theta^-=-\theta^+-1.
\]
The determinant line bundle and the trace connection on it are fixed, i.e. we fix a rank one connection $\zeta$ on $\mathcal O_{\mathbb P^1}(1)$ with residue $-1$ at $\infty$ and for any $(E,\nabla)\in{\rm Con}^{\theta}(\mathbb P^1,T)$ we have
\[
(\det E, {\rm tr}\nabla) = (\mathcal O_{\mathbb P^1}(1),\zeta).
\]

The eigenvalue $\theta$ is chosen so that we can construct a map 
\[
F\colon {\rm Con}^{\theta}(\mathbb P^1,T)\longrightarrow \contotal
\] 
which turns out to be an isomorphism (see \cite{Lame}), and which we now describe. Given 
\[
(E,\nabla)\in{\rm Con}^{\theta}(\mathbb P^1,T)
\]
we first take its pullback $(\pi^*E,\pi^*\nabla)$ via the $2$-cover $\pi:C\rightarrow \mathbb P^1$, the underlying vector bundle has determinant $\det(\pi^*E) = \mathcal O_C(2w_{\infty})$. After $\pi^*$ the eigenvalues are multiplied by $2$ and by performing an elementary transformation ${\rm elem}_{0,1,\lambda}$ over $w_0, w_1$ and $w_\lambda$, centered on directions corresponding to eigenvalues $1/2$,  we get 
\[
(E',\nabla') = {\rm elm}_{0,1,\lambda}(\pi^*E,\pi^*\nabla)
\]
with $\det E' = \mathcal O_C(-w_{\infty})$, recall that $w_0+w_1+w_{\lambda} \sim 3w_{\infty}$. The transformed connection $\nabla'$ has apparent singularities over $w_0, w_1$ and $w_\lambda$, thus they disappear after twist by a suitable rank one connection $\zeta_0$ on $\mathcal O_C(w_{\infty})$. More precisely, we may fix $\zeta_0$ with residues 
\[
-\frac{1}{2}, -\frac{1}{2}, -\frac{1}{2} \text{ and }\;\; \frac{1}{2}
\]
over $w_0, w_1,w_\lambda$ and $w_\infty$ to get
\[
F(E,\nabla):=(E',\nabla') \otimes (\mathcal O_C(w_{\infty}), \zeta_0) \in \contotal. 
\] 

We now describe the preimage of $\ZZ_1$ via $F$. For any $(E,\nabla)\in{\rm Con}^{\theta}(\mathbb P^1,T)$ the underlying vector bundle is always $E=\mathcal O_{\mathbb P^1}\oplus \mathcal O_{\mathbb P^1}(1)$, see \cite{Lame}. For each $i\in \{0,1,\lambda \}$ we let $p_i^+(\nabla)$ denote the parabolic direction with respect to the eigenvalue $1/4$ and $p_\infty^+(\nabla)$ to be the $\theta^+$ direction. Then define $(E, p^+(\nabla))$ the corresponding quasiparabolic bundle. Since $\nu \not\in \mathbb{Z}$, we have that $(E, p^+(\nabla) )$ is indecomposable and, in particular, there exist at most one direction lying in $\mathcal O_{\mathbb P^1}(1)$. Indeed, given two parabolic directions outside $\mathcal O_{\mathbb P^1}(1)$, we can choose an embedding of $\mathcal O_{\mathbb{P}^1}\rightarrow E$ passing through them. 


Then define
\[
A^i = \{ (E,\nabla)\in{\rm Con}^{\theta}(\mathbb P^1,T)\mid p_i^+(\nabla) \in \mathcal O_{\mathbb P^1}(1)\}
\]
and let $A^\infty$ be the locus of $(E,\nabla)\in{\rm Con}^{\theta}(\mathbb P^1,T)$ such that there is an embedding of $\mathcal O_{\mathbb P^1}$ passing through the three directions $ p_0^+(\nabla), p_1^+(\nabla), p_{\lambda}^+(\nabla)$. 

\begin{claim} 
We claim that $F$ sends $A^i$ to $\ZZ_1^i$.
\end{claim}

Let us assume $i=0$, other cases are similar. Pulling back $(E,\nabla)\in A^0$ to the elliptic curve we get 
\[
(\pi^*E,\pi^*\nabla) = (\mathcal O_C\oplus \mathcal O_C(2w_{\infty}), \pi^*\nabla)
\]
with one parabolic direction $p_0^+(\pi^*\nabla)$ lying in $\mathcal O_C(2w_{\infty})$. We may assume that $\mathcal O_C$ pass through $p_1^+(\nabla)$ and $p_{\lambda}^+(\nabla)$, thus after elementary transformation ${\rm elem}_{0,1,\lambda}$ we get 
\[
(E',\nabla') = (\mathcal O_C(-w_0)\oplus\mathcal O_C(2w_{\infty}-w_1-w_\lambda), \nabla')
\]
and twisting by $(\mathcal O_C(w_{\infty}), \zeta_0)$ we obtain 
\[
F(E,\nabla) = (\widetilde{E}, \nabla'\otimes \zeta_0) 
\]
where $\widetilde{E}\simeq L\oplus L^{-1}(w_\infty)$, with $L=\mathcal O_C(w_0)$. Consequently, $F(E,\nabla)\in \ZZ_1^0$. This proves the claim.

\section{Fuchsian systems with \texorpdfstring{$n+3$}{n+3} poles}\label{sect:fuchssyst}
Given $(E_1,\nabla)\in {\rm Con}^{\nu}$, we can associate a $\mathfrak {sl}_2$-connection on the trivial bundle $\mathcal O_C\oplus\mathcal O_C$ by performing an elementary transformation for a particular choice of directions over the $2$-torsion points $w_0$, $w_1$ and $w_\lambda$. We begin the section establishing this correspondence. The process creates new singularities which are apparent, i.e. they become regular points after one elementary transformation. 

We have that $t\in C$ is an apparent singular point for $\nabla$ if the residual part ${\rm Res}_t \nabla$ has $\{\frac{1}{2}, -\frac{1}{2}\}$ as eigenvalues and the $\frac{1}{2}$-eigenspace of ${\rm Res}_t \nabla$ is also invariant by the constant part of the connection matrix. 

Recall that $D'= w_0+w_1+w_\lambda + D$ is also reduced, see Remark \ref{rem:notation}. Let $(\nu_1, \dots, \nu_n) \in (\mathbb{C}^\ast)^n$ and fix an eigenvalue $\sigma$ by setting $\left(\frac{1}{2}, -\frac{1}{2}\right)$ over $w_0, w_1$ and $w_\lambda$, and $\left(\frac{\nu_j}{2}, -\frac{\nu_j}{2}\right)$ over $t_j$, $j= 1, \dots, n$, i.e. 
\begin{equation}\label{theta}
\sigma = \left(\frac{1}{2}, -\frac{1}{2}, \frac{1}{2}, -\frac{1}{2}, \frac{1}{2}, -\frac{1}{2} , \frac{\nu_1}{2}, -\frac{\nu_1}{2}, \dots, \frac{\nu_n}{2}, -\frac{\nu_n}{2}\right).
\end{equation}

We denote by ${\rm Syst}^{\sigma}(C,D')$ the moduli space of Fuchsian systems (i.e. logarithmic $\mathfrak {sl}_2$-connections on the trivial bundle $\mathcal O_C\oplus\mathcal O_C$) having pole divisor $D'$, eigenvalue $\sigma$ and such that: 
\begin{itemize}
\item the three singular points $w_0, w_1$ and $w_{\lambda}$ are apparent;
\item over $w_0$, $w_1$ and $w_\lambda$ the corresponding $\frac{1}{2}$-eigenspaces are $(1:0)$, $(1:1)$ and $(0:1)$ respectively.
\end{itemize}

\begin{figure}[ht]
	\centering
	\definecolor{qqqqff}{rgb}{0.,0.,1.}
	\definecolor{qqwuqq}{rgb}{0.,0.39215686274509803,0.}
	\definecolor{zzttqq}{rgb}{0.6,0.2,0.}
	\definecolor{ududff}{rgb}{0.30196078431372547,0.30196078431372547,1.}
	\definecolor{xdxdff}{rgb}{0.49019607843137253,0.49019607843137253,1.}
	\begin{tikzpicture}[line cap=round,line join=round,>=triangle 45,x=1.0cm,y=1.0cm]
		\begin{axis}[
			x=1.0cm,y=1.0cm,
			axis lines= left,
			axis line style={draw=none}, 
			ymajorgrids=false,
			xmajorgrids=false,
			xmin=2.4,
			xmax=10.4,
			ymin=-0.2,
			ymax=3.0,
			xtick= \empty, 
			ytick= \empty, 
			]
			\draw[line width=1.pt,color=zzttqq,smooth,samples=100,domain=2.5:8] plot(\x,{1.356976356976357E-4*(\x)^(6.0)-0.004306526806526807*(\x)^(5.0)+0.04784132534132534*(\x)^(4.0)-0.21022560772560772*(\x)^(3.0)+0.24674825174825174*(\x)^(2.0)+0.5092574092574093*(\x)}) node[anchor=west] {\color{black}$\mathcal{O}_C(w_0- w_\infty)$};
			\draw[line width=1.pt,color=qqwuqq,smooth,samples=100,domain=2.5:8] plot(\x,{2.0E-50*(\x)^(5.0)-6.4E-49*(\x)^(4.0)+7.3E-48*(\x)^(3.0)-3.5E-47*(\x)^(2.0)+6.0E-47*(\x)+2.0}) node[anchor= west] {\color{black}$\mathcal{O}_C(w_1- w_\infty)$};
			\draw[line width=1.pt,color=qqqqff,smooth,samples=100,domain=2.5:8] plot(\x,{1.48994523994524E-4*(\x)^(6.0)-0.0046662365412365415*(\x)^(5.0)+0.05146212583712584*(\x)^(4.0)-0.23185501998001998*(\x)^(3.0)+0.3543778906278906*(\x)^(2.0)-0.3574897324897325*(\x)+3.0}) node[anchor= west] {\color{black}$\mathcal{O}_C(w_\lambda - w_\infty)$};
			\draw [line width=1.5pt] (3.,3.)-- (3.,0.2) node[anchor=north] {$w_0$};
			\draw [line width=1.5pt] (5.,3.)-- (5.,0.2) node[anchor=north] {$w_1$};
			\draw [line width=1.5pt] (7.,3.)-- (7.,0.2) node[anchor=north] {$w_\lambda$};
			\draw[line width=1.pt,dash pattern=on 2pt off 3pt,smooth,samples=100,domain=2.5:8] plot(\x,{0.015714285714285715*(\x)^(4.0)-0.3238095238095238*(\x)^(3.0)+2.1871428571428573*(\x)^(2.0)-5.404761904761905*(\x)+5.0}) node[anchor= west] {\color{black}$\mathcal{O}_C$};
		\end{axis}
	\end{tikzpicture}
	\caption{Sections of $\mathbb{P}(E_1)$} \label{fig:sections}
\end{figure}
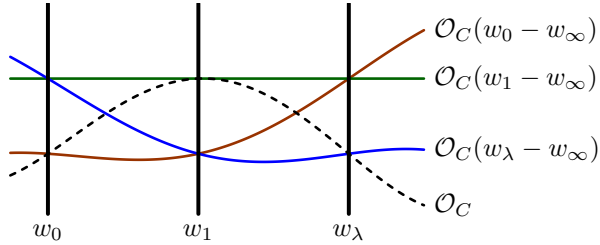

\begin{prop}\label{prop:consys}
There is an isomorphism of moduli spaces
\[
{\rm Con}^{\nu}\stackrel{\sim}{\longrightarrow} {\rm Syst}^{\sigma}(C,D')
\]
where $\displaystyle \nu_j = \nu^+_j-\nu^-_j$ for $j = 1, \dots, n$. 
\end{prop}

\begin{proof}
Let $(E_1,\nabla) \in {\rm Con}^\nu$. The subbundles $\mathcal{O}_C(w_0-w_\infty), \mathcal{O}_C(w_1-w_\infty), \mathcal{O}_C(w_\lambda-w_\infty)\subset E_1$ intersect as in Figure \ref{fig:sections} and we apply elementary transformation on these directions so that they become three disjoint copies of $\mathcal{O}_C(-w_\infty)$. Hence the elementary transformed of $(E_1,\nabla)$ is $(E,\nabla')$ where 
$E = \mathcal{O}_C(-w_\infty)\oplus \mathcal{O}_C(-w_\infty)$ and the eigenvalues of $\nabla'$ are the same as $\nabla$ at the $t_k$, $k=1,\dots, n$, and the eigenvalues at $w_0, w_1$ and $w_\lambda$ are equal to $(1,0)$.
Consider the connection $(\mathcal{O}_C(w_\infty), \xi)$ such that
\[
{\rm Res}_{w_0} \xi = Res_{w_1} \xi = Res_{w_\lambda} \xi= -\frac{1}{2} \quad \text{and} \quad {\rm Res}_{t_k} \xi = -\frac{\nu_k^++\nu_k^-}{2}.
\]
It exists since $\sum_{k=1}^n \nu_k^+ + \nu_k^- = -1$ by Fuchs relation. Therefore 
\[
(E_1,\nabla) \longmapsto (\mathcal{O}_C\oplus\mathcal{O}_C,\nabla'\otimes \xi) \in {\rm Syst}^{\sigma}(C,D')
\]
is our desired isomorphism.
\end{proof}

Consider the space $\BB$ of parabolic bundles $(E,{\bf p})$ over $(C,D')$ such that $E$ is the trivial bundle, the parabolic directions over $w_0,w_1$ and $w_\lambda$ are $(1:0)$, $(1:1)$ and $(0:1)$, respectively, and we let the parabolic directions over $D = t_1+\cdots +t_n$ vary. We have a natural identification $\BB=(\mathbb P^1)^n$ where each copy of $\mathbb{P}^1$ parametrizes parabolic directions over $t_j$. 

The elementary transformation used in the proof yields an alternative proof for the isomorphism 
${\rm Bun} \simeq (\mathbb{P}^1)^n$ in Proposition \ref{prop:1stchamb}. 
Indeed, $(E_1, {\bf p})$ corresponds to $(\mathcal{O}_C\oplus\mathcal{O}_C, {\bf p'})$ and, 
up to automorphism, we can fix the directions over $w_0, w_1$ and $w_\lambda$ so that it defines an element of $\BB$; in particular, this gives us the identification
\begin{eqnarray}\label{BB}
{\rm Bun} \simeq \BB. 
\end{eqnarray}

\subsection{The affine bundle of Fuchsian systems}

The moduli space ${\rm Syst}^{\sigma}(C,D')$ is an affine bundle of rank $n$ over $\BB$. We will describe its trivializations over the Zariski open sets 
\[
U_0 = \prod_{j=1}^{n} \left\{(z_j:w_j)\in \mathbb{P}^1\mid w_j=1\right\} \quad{\text and } \quad U_\infty = \prod_{j=1}^{n} \left\{(z_j:w_j)\in \mathbb{P}^1\mid z_j=1\right\},
\]
and then give the affine transition map.

In order to make explicit computations, we will consider that $C$ is given by the affine equation $y^2=x(x-1)(x-\lambda)$, with $\lambda\in\mathbb C\setminus\{0,1\}$, and $w_\infty\in C$ will be the point at infinity. In particular, the torsion points are $w_0 = (0,0)$, $w_1 = (1,0)$ and $w_\lambda = (\lambda, 0)$. To begin with, let us fix a basis of meromorphic one-forms with at most simple poles on $D'$. Consider the holomorphic one-form $ \omega = \frac{dx}{2y}=\frac{dy}{3x^2-2(1+\lambda)x+\lambda}$ and define 
\[
\displaystyle \phi_0 = \frac{(1-\lambda)x}{y}\omega, \;\displaystyle\phi_1 = \frac{-\lambda(x-1)}{y}\omega \text{ and } \theta_j = \frac{x_j(x_jx-\lambda)}{x_jy-y_jx}\omega, 
\]
where $t_j=(x_j,y_j)$. These $n+3$ one-forms give us the desired basis, 
and their residues (and constant term at $w_i$) are given as follows in terms of local coordinate $y$:
\begin{itemize}
\item at $w_0$, we have $x=O(y^2)$, and:
\begin{align*}
& \omega=\frac{dy}{\lambda}+O(y), && \phi_0=O(y), && \phi_1=\frac{dy}{y}+O(y), && \theta_j=-\frac{dy}{y}-\frac{y_j}{\lambda x_j}dy+O(y); 
\end{align*}
\item at $w_1$, we have $x=1+O(y^2)$, and:
\begin{align*}
&\omega=\frac{dy}{1-\lambda}+O(y), && \phi_0=\frac{dy}{y}+O(y),&& \phi_1=O(y), && \theta_j=\frac{x_j(x_j-\lambda)}{y_j}dy+O(y);
\end{align*}
\item at $w_\lambda$, we have $x=\lambda+O(y^2)$, and:
\begin{align*}
& \omega=\frac{dy}{\lambda(\lambda-1)}+O(y), && \phi_0=-\frac{dy}{y}+O(y), && \phi_1=-\frac{dy}{y}+O(y), && 
\theta_j=-\frac{x_j(x_j-1)}{y_j}dy+O(y);
\end{align*}
\item at $t_j$, the $1$-forms are holomorphic except for $\theta_j$ which has a simple pole with ${\rm Res}_{t_j}\theta_j=1$.
\end{itemize}
Therefore, any Fuchsian system $\nabla \in {\rm Syst}^{\sigma}(C,D')$ can be written as 
\[
\nabla = d+ \begin{pmatrix}	a_0 & b_0\\ c_0 & -a_0\end{pmatrix} \phi_0 + 
\begin{pmatrix}	a_1 & b_1\\	c_1 & -a_1\end{pmatrix} \phi_1+ 
\begin{pmatrix}	a_2 & b_2\\	c_2 & -a_2\end{pmatrix}\omega +
 \sum_{j=1}^{n}\begin{pmatrix}	\alpha_j & \beta_j\\ \gamma_j & -\alpha_j\end{pmatrix}\theta_j,
\]
and, as we impose apparent singularities at $w_0$, $w_1$ and $w_\lambda$, we get 
\begin{align*}
a_0 &= -\sum_{j=1}^n\alpha_j, & b_0 &= \frac{1}{2}+ \sum_{j=1}^n\alpha_j, &	 c_0 &= \frac{1}{2}-\sum_{j=1}^n\alpha_j, \\
a_1 &= \frac{1}{2} + \sum_{j=1}^n\alpha_j, & b_1 &=-\frac{1}{2}- \sum_{j=1}^n\alpha_j, & c_1&= \sum_{j=1}^n\gamma_j, \\
a_2 &= \sum_{j=1}^{n}\frac{y_j}{2}\left[ \frac{2\alpha_j+\beta_j-\gamma_j}{x_j-1} + \frac{\gamma_j}{x_j} -\frac{\beta_j}{x_j-\lambda}\right], & b_2 &= \sum_{j=1}^{n}\frac{y_j}{(x_j-\lambda)}\beta_j, & c_2 &= \sum_{j=1}^n \frac{y_j}{x_j}\gamma_j.
\end{align*}
In particular, $\nabla$ is completely determined by the residues 
\[
{\rm Res}_{t_j}(\nabla) = \begin{pmatrix}	\alpha_j & \beta_j\\ \gamma_j & -\alpha_j\end{pmatrix}.
\] 

\begin{prop} \label{prop:parfuchsSn}If $\nu_j \neq 0$ for $j=1, \dots, n$ then the map ${\rm Par}\colon {\rm Syst}^{\sigma}(C,D') \rightarrow S^n$ that associates to $\nabla$ the eigenspaces of ${\rm Res}_{t_j}(\nabla)$, $j=1, \dots, n$, is an isomorphism.
\end{prop}

\begin{proof}
We know that $\nabla$ is determined by ${\rm Res}_{t_j}(\nabla)$, $j=1, \dots, n$. On the other hand, 
\begin{align*}
\begin{pmatrix}
		\alpha_j & \beta_j\\
		\gamma_j & -\alpha_j
	\end{pmatrix} = \frac{1}{zv-wu}\begin{pmatrix}
		z &u\\
		w & v
	\end{pmatrix}\cdot \begin{pmatrix} \displaystyle
		\frac{\nu_j}{2} & 0 \\ 
		0 &\displaystyle -\frac{\nu_j}{2}
	\end{pmatrix} \cdot \begin{pmatrix}
		v & -u \\
		-w & z
	\end{pmatrix} 
\end{align*} 
where $(z:w)$ and $(u:v)$ are the eigenspaces associated to $\frac{\nu_j}{2}$ and $-\frac{\nu_j}{2}$, respectively. 
Note that $S = \mathbb{P}^1 \times \mathbb{P}^1 \setminus \{ zv-wu =0\}$, hence the isomorphism is clear.
\end{proof}

\begin{remark}\label{rem:consys}
Note that we can give an alternative proof of Theorem \ref{thm:parconSn} using that Proposition \ref{prop:consys} and Proposition \ref{prop:parfuchsSn} combined give an isomorphism by the composition 
 \[
 {\rm Con}^{\nu}\longrightarrow {\rm Syst}^{\sigma}(C,D') \longrightarrow S^n.
 \]
\end{remark}





\subsubsection{Trivialization over $U_0$}\label{section: U_0} 
Given ${\bf z} = (z_1, \dots, z_n)\in U_0$ we define $\nabla_0$ as the connection which has $(z_j:1)$ as eigenspaces corresponding to $\frac{\nu_j}{2}$ and has $(1:0)$ as eigenspaces corresponding to $-\frac{\nu_j}{2}$. It can be written as 
\[
\nabla_0 = d+ \begin{pmatrix}
a_0 & b_0\\
c_0 & -a_0
\end{pmatrix} \phi_0 + \begin{pmatrix}
a_1 & b_1\\
0 & -a_1
\end{pmatrix} \phi_1+ \begin{pmatrix}
a_2 & b_2\\
0 & -a_2
\end{pmatrix}\omega + \sum_{j=1}^{n}\frac{\nu_j}{2}\begin{pmatrix}
-1 & 2z_j\\
0 & 1
\end{pmatrix}\theta_j.
\]
with
\[
\begin{array}{lll}
\displaystyle a_0 = \sum_{j=1}^n\frac{\nu_j}{2} &
\displaystyle b_0 = \frac{1}{2}- \sum_{j=1}^n\frac{\nu_j}{2} &
\displaystyle c_0 = \frac{1}{2}+\sum_{j=1}^n\frac{\nu_j}{2}\\
\displaystyle a_1 = \frac{1}{2} -\sum_{j=1}^n\frac{\nu_j}{2}&
\displaystyle b_1 =-\frac{1}{2}+\sum_{j=1}^n\frac{\nu_j}{2} \\
\displaystyle a_2 = \sum_{j=1}^{n}\frac{\nu_jy_j}{2} \left[ \frac{z_j-1}{x_j-1} -\frac{z_j}{x_j-\lambda}\right]&
\displaystyle b_2 = \sum_{j=1}^{n} \nu_j y_j \frac{z_j}{(x_j-\lambda)}
\end{array}
\]

Now if $\nabla$ is any connection which has $ {\bf z}$ as positive eigenspaces then the difference $\Theta=\nabla - \nabla_0$ is a Higgs field which is nilpotent with respect to ${\bf z}$.
This means that $\Theta$ is a strongly parabolic Higgs field over $(\mathcal O_C\oplus\mathcal O_C, {\bf z})$. 
We shall fix a basis $\{\Theta_1^0,\cdots, \Theta_n^0\}$ for the space of strongly parabolic Higgs fields such that 
\[
{\rm Res}_{t_j}\Theta_j^0 = \begin{pmatrix}
	z_j & -z_j^2\\
	1 & -z_j
\end{pmatrix}
\]
and ${\rm Res}_{t_i}\Theta_j^0 = 0$ for $i\neq j$. So, we define 
\begin{align*}
\Theta_j^0 & = \begin{pmatrix}
-z_j & z_j\\
-z_j & z_j
\end{pmatrix} \phi_0 + \begin{pmatrix}
z_j & -z_j\\
1 & -z_j
\end{pmatrix} \phi_1+ A_j\cdot\omega + \begin{pmatrix}
z_j & -z_j^2\\
1 & -z_j
\end{pmatrix}\theta_j
\end{align*}
where
\[
 A_j = \begin{pmatrix}
\frac{y_j}{2}\left[ \frac{-(z_j-1)^2}{x_j-1}+\frac{1}{x_j}+ \frac{z_j^2}{x_j-\lambda}\right] & -\frac{y_jz_j^2}{x_j-\lambda}\\
\frac{y_j}{x_j} & -\frac{y_j}{2}\left[ \frac{-(z_j-1)^2}{x_j-1}+\frac{1}{x_j}+ \frac{z_j^2}{x_j-\lambda}\right]
\end{pmatrix}.
\]
This matrix $A_j$ has been chosen to assure apparent singularities over $w_0,w_1$ and $w_{\lambda}$. 

Any $\mathfrak{sl}_2$-connection $(E,\nabla)\in{\rm Syst}^{\sigma}(C,D')$ having $(z_j:1)$ as parabolic direction (over $t_j$) corresponding to $\frac{\nu_j}{2}$ can be written as 
\[
\nabla = \nabla_0 + r_j\Theta_j^0
\]
for suitable $r_j\in\mathbb C$. It has $(z_jr_j-\nu_j:r_j)$ as the complementary direction corresponding to $-\frac{\nu_j}{2}$. 

\subsubsection{Trivialization over $U_\infty$}
Similarly, we define $\nabla_\infty$ having $(1:w_j)$ as eigenspaces corresponding to $\frac{\nu_j}{2}$ and $(0:1)$ as eigenspaces corresponding to $-\frac{\nu_j}{2}$:
\[
\nabla_\infty = d+ \begin{pmatrix}
a_0 & b_0\\
c_0 & -a_0
\end{pmatrix} \phi_0 + \begin{pmatrix}
a_1 & b_1\\
c_1 & -a_1
\end{pmatrix} \phi_1+ \begin{pmatrix}
a_2 & 0\\
c_2 & -a_2
\end{pmatrix}\omega + \sum_{j=1}^{n}\frac{\nu_j}{2} \begin{pmatrix} \displaystyle
1 &\displaystyle 0 \\\displaystyle
2w_j & \displaystyle-1
\end{pmatrix}\theta_j
\]
with 
\[
\begin{array}{lll}
\displaystyle a_0 = -\sum_{j=1}^n\frac{\nu_j}{2} &
\displaystyle b_0 = \frac{1}{2}+ \sum_{j=1}^n\frac{\nu_j}{2}&
\displaystyle c_0 = \frac{1}{2}-\sum_{j=1}^n\frac{\nu_j}{2}\\
\displaystyle a_1 = \frac{1}{2} + \sum_{j=1}^n\frac{\nu_j}{2}&
\displaystyle b_1 =-\frac{1}{2}- \sum_{j=1}^n\frac{\nu_j}{2} &
\displaystyle c_1= \sum_{j=1}^n\nu_jw_j \\
\displaystyle a_2 = \sum_{j=1}^{n}\frac{\nu_jy_j}{2}\left[ \frac{(1-w_j)}{x_j-1} + \frac{w_j}{x_j} \right]& &
\displaystyle c_2 = \sum_{j=1}^n \frac{y_j}{x_j}\nu_jw_j
\end{array}
\]
And the strongly parabolic Higgs fields are defined by
\begin{align*}
\Theta_j^\infty & = \begin{pmatrix}
-w_j & w_j\\
-w_j & w_j
\end{pmatrix} \phi_0 + \begin{pmatrix}
w_j & -w_j\\
w_j^2 & -w_j
\end{pmatrix} \phi_1+ B_j\cdot\omega + \begin{pmatrix}
w_j & -1\\
w_j^2 & -w_j
\end{pmatrix}\theta_j
\end{align*}
where
\[
B_j=\begin{pmatrix}
\frac{y_j}{2}\left[ \frac{-(1-w_j)^2}{x_j-1} + \frac{w_j^2}{x_j} +\frac{1}{x_j-\lambda}\right] & \frac{-y_j}{(x_j-\lambda)}\\
\frac{y_jw_j^2}{x_j} & -\frac{y_j}{2}\left[ \frac{-(1-w_j)^2}{x_j-1} + \frac{w_j^2}{x_j} +\frac{1}{x_j-\lambda}\right]
\end{pmatrix}.
\]

Given a Fuchsian system $\nabla\in{\rm Syst}^{\sigma}(C,D')$ with $(1:w_j)$ as parabolic direction (over $t_j$) corresponding to $\frac{\nu_j}{2}$, we can write
\[
\nabla = \nabla_{\infty} + s_j\Theta_j^{\infty}
\]
for suitable $s_j\in\mathbb C$. The complementary direction corresponding to $-\frac{\nu_j}{2}$ is $(s_j:w_js_j+\nu_j)$. 

\subsubsection{Transition matrix}
From the trivializations on $U_0$ and $U_\infty$ we may compute a transition affine transformation for the affine bundle ${\rm Syst}^{\sigma}(C,D')$. Since $U_0\cup U_{\infty}$ covers the basis $\BB\simeq (\mathbb P^1)^n$ minus a subvariety of codimension two, then the bundle structure of ${\rm Syst}^{\sigma}(C,D')$ is determined by this affine transformation.

In order to make this affine transformation explicit we note that 
\[
\Theta_j^\infty = w_j^2 \Theta_j^0\left(\frac{1}{w_j}\right),
\]
and
\[
\nabla_0\left(\frac{1}{w}\right) = d+ \begin{pmatrix}
a_0 & b_0\\
c_0 & -a_0
\end{pmatrix} \phi_0 + \begin{pmatrix}
a_1 & b_1\\
0 & -a_1
\end{pmatrix} \phi_1+ \begin{pmatrix}
a_2 & b_2\\
0 & -a_2
\end{pmatrix}\omega + \sum_{j=1}^{n}\frac{\nu_j}{2}\begin{pmatrix}
-1 & \frac{2}{w_j}\\
0 & 1
\end{pmatrix}\theta_j.
\]
with
\[
\begin{array}{lll}
\displaystyle a_0 = \sum_{j=1}^n\frac{\nu_j}{2}&
\displaystyle b_0 = \frac{1}{2}- \sum_{j=1}^n\frac{\nu_j}{2}&
\displaystyle c_0 = \frac{1}{2}+\sum_{j=1}^n\frac{\nu_j}{2}\\
\displaystyle a_1 = \frac{1}{2} -\sum_{j=1}^n\frac{\nu_j}{2}&
\displaystyle b_1 =-\frac{1}{2}+\sum_{j=1}^n\frac{\nu_j}{2} \\
\displaystyle a_2 = \sum_{j=1}^{n}\frac{y_j\nu_j}{2w_j}\left[ \frac{1-w_j}{x_j-1} -\frac{1}{x_j-\lambda}\right]&
\displaystyle b_2 = \sum_{j=1}^{n}\frac{y_j}{(x_j-\lambda)}\frac{\nu_j }{w_j}
\end{array}
\]
Hence 
\begin{equation}\label{eq:splitting}
\nabla_\infty = \nabla_0\left(\frac{1}{w_1}, \dots, \frac{1}{w_n}\right) + \sum_{j=1}^n \nu_jw_j\Theta_j^0\left(\frac{1}{w_j}\right).
\end{equation}
Over the intersection $U_0\cap U_\infty$ we have
\begin{align*}
(s_1, \dots, s_n) &\mapsto \nabla_\infty(w_1,\dots,w_n) + \sum_{j=1}^n s_j\Theta_j^\infty(w_j) = \\ &=\nabla_0\left(\frac{1}{w_1}, \dots, \frac{1}{w_n}\right) + \sum_{j=1}^n \left(\nu_jw_j+ s_jw_j^2\right)\Theta_j^0\left(\frac{1}{w_j}\right) =\\
& = \nabla_0(z_1, \dots, z_n) + \sum_{j=1}^n \left(\nu_jw_j+ s_jw_j^2\right)\Theta_j^0(z_j) =\\
& = \nabla_0(z_1, \dots, z_n) + \sum_{j=1}^n r_j\Theta_j^0(z_j)
\end{align*}
which gives us the transition affine transformation
\begin{equation}\label{eq:aftrans}
	\begin{pmatrix}
		r_1 \\
		\vdots\\
		r_n
	\end{pmatrix} = 
	\begin{pmatrix}
		w_1^2 & \dots & 0 \\
		\vdots &\ddots & \vdots \\
		0 &\dots & w_n^2
	\end{pmatrix}\cdot
	\begin{pmatrix}
		s_1 \\
		\vdots\\
		s_n
	\end{pmatrix} + 
	\begin{pmatrix}
		\nu_1w_1 \\
		\vdots\\
		\nu_nw_n
	\end{pmatrix}.
\end{equation}
Note that if we choose new trivializations such that $r'_j = \frac{\rho_j}{\nu_j}r_j$ and $s'_j =\frac{\rho_j}{\nu_j}s_j$, this transition map becomes
\[
\begin{pmatrix}
	r_1'\\
	\vdots\\
	r_n'
\end{pmatrix} = 
\begin{pmatrix}
	w_1^2 & \dots & 0 \\
	\vdots &\ddots & \vdots \\
	0 &\dots & w_n^2
\end{pmatrix}\cdot
\begin{pmatrix}
	s_1' \\
	\vdots\\
	s_n'
\end{pmatrix} + 
\begin{pmatrix}
	\rho_1w_1 \\
	\vdots\\
	\rho_nw_n
\end{pmatrix}.
\]
Hence the affine bundles ${\rm Con}^\nu\xrightarrow{\pi_+} {\rm Bun}$ are all isomorphic, not depending on $\nu$; this agrees with Theorem \ref{thm:parconSn}. The transition map \eqref{eq:aftrans} can also be written as
\begin{equation}\label{eq:transition}
	\begin{pmatrix}
		\lambda\\
		r_1 \\
		\vdots\\
		r_n
	\end{pmatrix} =  
	\begin{pmatrix}
		1 & 0 & \dots & 0 \\
		\nu_1w_1 & w_1^2 & \dots & 0 \\
		\vdots & \vdots &\ddots & \vdots \\
		\nu_n w_n& 0 &\dots & w_n^2
	\end{pmatrix}\cdot
	\begin{pmatrix}
		\lambda \\
		s_1 \\
		\vdots\\
		s_n 
	\end{pmatrix} 
\end{equation}
and can be thought, owing to \eqref{BB}, as the transition matrix of a vector bundle $\mathcal{E}$ over ${\rm Bun}$ that parameterizes logarithmic $\lambda$-connections: triples $(E_1, \lambda, \nabla)$ where $\nabla \colon E_1 \rightarrow E_1 \otimes \Omega_C(D)$ and $\lambda \in \mathbb{C}$ such that
\[
\nabla(fs) = \lambda s \otimes df + f\nabla(s),
\]
see \cite[Section 4]{SimpLcon}. If $\lambda\neq 0$ then $\frac{1}{\lambda}\nabla$ is an usual logarithmic connection and if $\lambda=0$ then $\nabla$ is a Higgs field. We note that the complement of $U_0\cup U_{\infty}$ is a subvariety of codimension two, hence the transition matrix given in \eqref{eq:transition} determines $\mathcal{E}$.

The projectivization of $\mathcal{E}$ provides a compactification $\overline{{\rm Con}^\nu}$ for the moduli space ${\rm Con}^\nu$; in the boundary we get the projectivized moduli space of Higgs fields
\[
\mathbb{P}\Hig \simeq \overline{{\rm Con}^\nu} \setminus {\rm Con}^\nu.
\]
The inclusion $\mathbb{P}\Hig \hookrightarrow \mathbb{P}(\mathcal{E})$ comes from the natural extension
\begin{equation}\label{eq:extension}
	0\longrightarrow T^*{\rm Bun} {\longrightarrow} \mathcal E {\longrightarrow} \mathcal O_{{\rm Bun}} \longrightarrow 0
\end{equation}
where the first map is the inclusion of Higgs fields and the last is the projection $(E_1,\lambda, \nabla) \mapsto \lambda$. We saw that $\mathcal{E}$ does not depend on $\nu$. However, this extension, hence the inclusion $\mathbb{P}\Hig \hookrightarrow \mathbb{P}(\mathcal{E})$, is determined by the exponents $\nu$.

\begin{thm}\label{thm:extension}
The moduli space ${\rm Con}^\nu$ has compactification  $\overline{{\rm Con}^{\nu}}=\mathbb{P}(\mathcal{E})$, where the boundary divisor is isomorphic to $\mathbb{P}\Hig$, the projectivization of the space of Higgs fields on $E_1$. Moreover, the inclusion $\mathbb{P}\Hig\hookrightarrow \mathbb{P}(\mathcal{E})$ is determined, up to automorphisms of $\mathbb{P}(\mathcal{E})$, by $\left(\nu_{1},\dots, \nu_{n}\right)$.
\end{thm}

\begin{proof}
The inclusion $\phi\colon\mathbb{P}\Hig \hookrightarrow \mathbb{P}(\mathcal{E})$ comes from the projectivization of \eqref{eq:extension}. We will show that the isomorphism class of this extension is determined by $\left(\nu_{1},\dots, \nu_{n}\right)$. This is equivalent to determining $\phi$, up to composing with an automorphism of $\mathbb{P}(\mathcal{E})$ as a $\mathbb{P}^n$-bundle over ${\rm Bun}$. The proof will follow by  computing the extension class $e^\nu \in H^{1}({\rm Bun}, T^*{\rm Bun})$,  see \cite[p. 185]{At}.

Over $U_0$, the map $\mathcal{O}_{{\rm Bun}} \rightarrow \mathcal{E}^{\nu}$, given by $h\mapsto h\nabla_0$, defines a splitting; the same is true for $\nabla_{\infty}$ over $U_\infty$. Therefore $e^\nu$ is represented, on $U_0\cap U_\infty$, by
\[
\nabla_\infty - \nabla_0 = \sum_{j=1}^n \nu_jw_j\Theta_j^0\left(\frac{1}{w_j}\right),
\]
see the equation in display \eqref{eq:splitting}. Any cohomologous cocycle must be
\begin{align*}
	&&\sum_{j=1}^n \nu_jw_j\Theta_j^0\left(\frac{1}{w_j}\right) + \sum_{j = 1}^n b_j(w_1, \dots w_n)\Theta_j^\infty(w_j) - \sum_{j = 1}^n a_j\left(\frac{1}{w_1}, \dots, \frac{1}{w_n}\right) \Theta_j^0\left(\frac{1}{w_j}\right) =\\
		&&= \sum_{i=1}^n \left[\nu_jw_j + b_j(w_1, \dots w_n)w_j^2 - a_j\left(\frac{1}{w_1}, \dots, \frac{1}{w_n}\right)\right] \Theta_j^0\left(\frac{1}{w_j}\right)
\end{align*}
for some holomorphic functions $a_i$ and $b_i$. In particular, the linear terms $\nu_jw_j$ are left unchanged, whence the extension is completely determined by $\left(\nu_{1},\dots, \nu_{n}\right)$.
\end{proof}

Returning to the affine bundle ${\rm Con}^\nu\xrightarrow{\pi_+} {\rm Bun}$, we have seen that it does not distinguish $\nu$. However, if we consider ${\rm Con}^\nu$ with the inherited symplectic structure from $\contotal$, the picture is different; this will be the topic of the next section. The trivializations in \eqref{eq:aftrans} will be important to bring the symplectic form to the Darboux normal form.

\section{Symplectic structure}\label{sect:symp}
It follows from the work of Iwasaki \cite{Iwasaki91}
that moduli spaces of logarithmic connections on curves carry a natural symplectic structure. 
In fact, there exists a $2$-form defined on a larger moduli space where the poles are allowed
to move on $C$ inducing a Poisson structure whose kernel defines the isomonodromic 
deformations, therefore a Hamiltonian system.
Moreover, through the Riemann-Hilbert correspondence, Iwasaki proved
in \cite{Iwasaki92} that this $2$-form coincides with the symplectic structure on moduli spaces of representations
that originated through the works of  Atiyah-Bott and Goldman. In fact, Iwasaki considers moduli spaces of Sturm-Liouville operators rather than connections. In order to define the symplectic structure globally on our moduli space, the approach of  Arinkin-Lysenko in \cite{AL} and Iwasaki in \cite{Iwasaki92} is more convenient for us. We briefly recall how the tangent space and the symplectic structure are constructed there. Later we will give an explicit formula for the $2$-form on an open set of our moduli space. 

First we recall the following result from \cite[Section 15]{IY}.

\begin{lemma}\label{lem:LocalModels}
	Assume that the connection $(E,\nabla)$ has a simple pole at $t\in C$ 
	with eigenvalues $\{\nu^+,\nu^-\}$, and let $z\colon (C,t)\rightarrow(\mathbb C,0)$ be a local coordinate.
	Then there exists a local trivialization of the vector bundle $\Phi\colon E\vert_{(C,t)}\rightarrow (\mathbb C,0)\times \mathbb C^2$
	such that $\Phi_*\nabla$ is one of the following models:
	\begin{equation}\label{eq:LocalModelDiag}
		d+\begin{pmatrix}\nu^+&0\\ 0&\nu^-\end{pmatrix}\frac{dx}{x},
	\end{equation}
	or, in the resonant case $\{\nu^+,\nu^-\}=\{\nu,\nu+n\}$, $n\in\mathbb Z_{\ge0}$
	\begin{equation}\label{eq:LocalModelRes}
		d+\begin{pmatrix}\nu&x^n\\ 0&\nu+n\end{pmatrix}\frac{dx}{x}.
	\end{equation}
	Moreover, in the resonant case $\nu^+-\nu^-\in \mathbb{Z}$, the model \eqref{eq:LocalModelRes} is generic among the two, i.e. occurs for an open set of the deformations of $(E,\nabla)$.
\end{lemma}

Next we want to describe the tangent space of $\contotal$ at a (generic) connection $(E,\nabla)\in \contotal$. Consider a first order deformation $(E_\epsilon,\nabla_\epsilon)$ of $(E,\nabla)$; we briefly recall the definition to set up notation. Denote $R= \mathbb{C}[\epsilon]$, with $\epsilon^2 =0$, the ring of dual numbers. Then $E_\epsilon$ is a vector bundle over $C\times {\rm Spec}(R)$, flat over ${\rm Spec}(R)$,  such that the restriction induced by $R \twoheadrightarrow R/(\epsilon) = \mathbb{C}$ gives $E_\epsilon \otimes_{R} \mathbb{C} = E$. Moreover there exist a covering $\{U_i\}$ of $C$ and isomorphisms 
\begin{equation}\label{eq:isodef}
    \Phi_i \colon E_\epsilon|_{U_i \times {\rm Spec}(R)}  \stackrel{\sim}{\longrightarrow} E|_{U_i} \otimes_{\mathbb{C}} R,
\end{equation}
such that $\Phi_i\otimes 1\colon E_\epsilon|_{U_i \times {\rm Spec}(R)} \otimes_{R} \mathbb{C} = E|_{U_i} \rightarrow E|_{U_i}$ is the identity. On the other hand, $\nabla_\epsilon$ is a relative connection, i.e. an $R$-linear map $\nabla_\epsilon\colon E_\epsilon \rightarrow E_\epsilon \otimes_{\mathbb{C}} \Omega_C(D)$ satisfying the relative Leibniz rule, such that its restriction to $E_\epsilon \otimes_{R} \mathbb{C} = E$ is $\nabla$. Up to refining $\{U_i\}$, the isomorphisms $\Phi_i$ plus trivializations of $E$ amount to writing 
\begin{equation}\label{eq:trivnablaep}
    \nabla_\epsilon |_{U_i \times {\rm Spec}(R)} = d+ A_i^0 + \epsilon A_i^1
\end{equation}
where $\nabla|_{U_i} = d+ A_i^0$  and $A_i^1 \in Hom( E|_{U_i} , E\otimes\Omega_C(D)|_{U_i} )$. 
We are specially interested in deformations that preserve the trace and residual eigenvalues, which impose conditions on $A_i^1$.



We can refine $\{U_i\}$ further so that the restriction of $\nabla$ is either trivial or can be normalized like in Lemma \ref{lem:LocalModels}. Since the local models \eqref{eq:LocalModelDiag} with $\nu^+-\nu^-\not\in\mathbb Z$, and (\ref{eq:LocalModelRes})
are stable under small deformations with fixed eigenvalues, we deduce that there exist bundle automorphisms 
\[
\Psi_i \colon E|_{U_i} \otimes_{\mathbb{C}} R \stackrel{\sim}{\longrightarrow} E|_{U_i} \otimes_{\mathbb{C}} R
\]
such that $\nabla_\epsilon |_{U_i \times {\rm Spec}(R)} = \Psi_i^\ast (\nabla \otimes 1)$. Therefore we may assume that the maps $\Phi_i$ of \eqref{eq:isodef} are such that $A_i^1 = 0$ in \eqref{eq:trivnablaep}.


Over $U_{i}\cap U_{j}$ the comparison $\Phi_{ij}=\Phi_i\circ\Phi_j^{-1}$ defines an automorphism of $E\vert_{U_{i}\cap U_{j}}\otimes_{\mathbb{C}} R $ preserving the connection $\nabla\vert_{U_{i}\cap U_{j}} \otimes 1$. By construction $\Phi_{ij} = \text{id}+\epsilon\phi_{ij}$ and we get 
\begin{align*}
    (\nabla \otimes 1)(\Phi_{ij}(s_0+\epsilon s_1)) & = (\nabla \otimes 1) (s_0+\epsilon( s_1 + \phi_{ij}(s_0))) = \nabla(s_0)+\epsilon[ \nabla(s_1) + \nabla(\phi_{ij}(s_0))]\\
    \rotatebox{90}{=} \quad \quad \quad \quad \quad & \\
    \Phi_{ij}((\nabla \otimes 1)(s_0+\epsilon s_1)) & =  \Phi_{ij}(\nabla(s_0)+\epsilon \nabla(s_1)) = \nabla(s_0)+\epsilon[ \nabla(s_1)  + \phi_{ij}(\nabla(s_0))].
\end{align*}
Hence $\nabla(\phi_{ij}(s_0)) = \phi_{ij}(\nabla(s_0)) $, i.e., $\phi_{ij}$ preserves $\nabla\vert_{U_{i}\cap U_{j}}$. 

Since we are assuming that the trace connection is fixed, $\det(\Phi_i)$ comes from a global isomorphism $\rho\colon \det(E_\epsilon) \xrightarrow{\sim} \det(E)\otimes_{\mathbb{C}}R$. Hence we have that $1 = \det \Phi_{ij} = 1 + \epsilon \, {\rm tr}\phi_{ij}$. Then ${\rm tr}\phi_{ij} =0$ and we get an element $\{\phi_{ij}\}$ of $H^1(\mathfrak{sl}(E,\nabla))$, 
the sheaf of trace-free endomorphisms preserving $\nabla$. Therefore 
\[
T_{(E,\nabla)} \contotal \simeq H^1(\mathfrak{sl}(E,\nabla)).
\]



Observe that $\mathfrak{sl}(E,\nabla)$ is a subsheaf of $\mathfrak{sl}(E,\textbf{p})$ the sheaf of endomorphisms of the quasi-parabolic bundle $(E,\textbf{p})$ with trace zero. Also note that $\nabla$ induces a natural (logarithmic) connection on the vector bundle $\mathfrak{sl}(E)$, namely $\nabla^1\colon \phi\mapsto\nabla\circ\phi-\phi\circ\nabla$. We then construct a complex of sheaves of abelian groups 
\begin{align}\label{eq:shortexactnabla1}
0 \longrightarrow \mathfrak{sl}(E,\nabla) \longrightarrow \mathfrak{sl}(E,\textbf{p}) \stackrel{\nabla^1}{\longrightarrow} \mathfrak{sl}(E,\textbf{p})\otimes\Omega^1_C(D)^{\text{nil}}\longrightarrow 0
\end{align}
where the rightmost member corresponds to parabolic Higgs fields with nilpotent residues. The two sheaves on the right correspond to $\mathcal F^0$ and $\mathcal F^1$ in \cite{IIS,Inaba,IS,Komyo}. Note that $\nabla^1$ is not $\mathcal{O}_C$-linear and $\mathfrak{sl}(E,\nabla)$ is not an $\mathcal{O}_C$-module. Next we show that for a generic connection the deformation complex \eqref{eq:shortexactnabla1} is indeed exact.

\begin{lemma}\label{lem:seqdef}
For a generic connection $(E,\nabla)$ in the moduli space, the sequence in display \eqref{eq:shortexactnabla1} is exact.
\end{lemma}

\begin{proof} First note that the kernel of $\nabla^1$ is precisely $\mathfrak{sl}(E,\nabla)$ and the leftmost map is the inclusion. Thus we only need to show that $\nabla^1$ is surjective, for a generic $\nabla$. It will follow from a direct computation with local models.

Outside the polar locus, $(E,\nabla)$ is locally trivial and the map $\nabla^1$ may be written as 
\[
\begin{pmatrix}a&b\\ c&-a\end{pmatrix}\longmapsto \begin{pmatrix}da&db\\ dc&-da\end{pmatrix}
\]
so that local surjectivity of $\nabla^1$ at a regular point follows from local integration of holomorphic $1$-forms.
On the other hand, at a non resonant pole, i.e. $\nu^+-\nu^-\not\in\mathbb{Z}$, 
the connection $(E,\nabla)$ is locally like \eqref{eq:LocalModelDiag}
and the map $\nabla^1$ may be written as 
\[
\begin{pmatrix}a&b\\ c&-a\end{pmatrix}\longmapsto \begin{pmatrix}da&db+(\nu^+-\nu^-)b\frac{dx}{x}\\ dc+(\nu^--\nu^+)c\frac{dx}{x}&-da\end{pmatrix}
\]
where $c(0)=0$ (parabolic condition); for the local surjectivity, we just have to check that a matrix of $1$-forms
\[
\begin{pmatrix}\alpha&\beta\\ \gamma&-\alpha\end{pmatrix}
\]
is in the image if and only if $\alpha$ and $\gamma$ are holomorphic (parabolic and nilpotent condition), 
and $\beta$ has a simple pole, which is straightforward.
Finally, consider a resonant pole of the form \eqref{eq:LocalModelRes}.
Then the map $\nabla^1$ may be written as 
\[
\begin{pmatrix}a&b\\ c&-a\end{pmatrix}\longmapsto \begin{pmatrix}da+cx^n\frac{dx}{x}&db-(nb+2ax^n)\frac{dx}{x}\\ dc+nc\frac{dx}{x}&-da-cx^n\frac{dx}{x}\end{pmatrix}.
\]
To prove the surjectivity in that case, we have to successively integrate 
\[
\left\{\begin{matrix} 
dc+nc\frac{dx}{x}&=\gamma\\
da+cx^n\frac{dx}{x}&=\alpha\\
db-(nb+2ax^n)\frac{dx}{x}&=\beta
\end{matrix}\right.
\]
In the first equation, $c(x)$ can be derived such that $c(0)=0$ (parabolicity); therefore we can integrate
the second one and find $a(x)$ up to a constant; finally, for a good choice of $a(0)$, the last equation admits a solution $b(x)$.
\end{proof}

Hereafter we will assume that $(E,\textbf{p})$ is $\mu$-stable for some weight $\mu$ and we will denote by $\mathcal{B}$ the corresponding moduli space of parabolic bundles. 

Note that $(E,\textbf{p})$ is simple, i.e. $H^0(\mathfrak{sl}(E,\textbf{p}))=0$, then we get $H^1(\mathfrak{sl}(E,\textbf{p})\otimes\Omega^1_C(D)^{\text{nil}})=0$
by Serre duality. The long exact sequence of cohomology of \eqref{eq:shortexactnabla1} gives
\begin{equation}\label{eq:extractlongnabla1}
	0 \longrightarrow H^0(\mathfrak{sl}(E,\textbf{p})\otimes\Omega^1_C(D)^{\text{nil}}) \longrightarrow H^1(\mathfrak{sl}(E,\nabla)) \rightarrow H^1(\mathfrak{sl}(E,\textbf{p})) \longrightarrow 0
\end{equation}
In the middle we get the tangent space to the moduli space of connections at $(E,\nabla)$. On the left-hand side, we get the space of Higgs fields, which does not depend on $\nabla$. And on the right-hand side we have $H^1(\mathfrak{sl}(E,\textbf{p}))$ that is the tangent space to $\mathcal{B}$ at $(E,\bf{p})$. In particular, the image of $\{\phi_{ij}\}$ in $H^1(\mathfrak{sl}(E,\textbf{p}))$, where we omit the connection and just keep track of the parabolic data, corresponds to the underlying first order deformation of the parabolic bundle.

We can check that the map 
$H^0(\mathfrak{sl}(E,\textbf{p})\otimes\Omega^1_C(D)^{\text{nil}})\hookrightarrow H^1(\mathfrak{sl}(E,\nabla))$ of \eqref{eq:extractlongnabla1}
is the derivative of the natural action $\nabla\mapsto\nabla+\theta$ for a global parabolic Higgs field
on the affine space of parabolic connections on $(E,\textbf{p})$.
Indeed, if we write $\nabla_\epsilon=\nabla+\epsilon\theta$ and $\Phi_i=\text{id}+\epsilon\phi_i $,
then we get 
\begin{align*}
	(\nabla+\epsilon\theta)|_{U_i} = (\Phi_i)^*\nabla &=(\text{id}-\epsilon\phi_i)\circ\nabla\circ(\text{id}+\epsilon\phi_i) = \\
	&=\nabla+\epsilon(\nabla\circ\phi_i-\phi_i\circ\nabla) =\nabla + \epsilon\nabla^1(\phi_i)
\end{align*}
The first order deformation is therefore encoded in $\{\phi_j-\phi_i\}$, which, by construction, has zero image in $H^1(\mathfrak{sl}(E,\textbf{p}))$. But the image of $\theta$ in $H^1(\mathfrak{sl}(E,\nabla))$ given by 
\eqref{eq:extractlongnabla1} gives exactly the same formula: we first integrate $\nabla^1\phi_i=\theta\vert_{U_i}$
and then associate $\{\phi_j-\phi_i\}$.

One can define a bilinear map on $T_{(E,\nabla)} \contotal$ as follows:
\begin{equation}\label{eq:DefSymplecticInfinitesimal}
	\begin{tikzcd}[row sep=1.pt, ampersand replacement=\&]
		H^1(\mathfrak{sl}(E,\nabla))\times H^1(\mathfrak{sl}(E,\nabla)) \arrow[r, "\displaystyle \omega_{\rm Con}"] \& H^2(\mathbb C)\simeq\mathbb C;\\
		( \{\phi_{ij}\} , \{\psi_{ij}\} ) \arrow[r, maps to,shorten <= 2.8em] \& \{u_{ijk}={\rm tr}(\phi_{ij}\psi_{jk})\}.
	\end{tikzcd}
\end{equation}
One can choose the covering $\{U_i\}$ such that $(E,\nabla)$ is trivial over $U_i\cap U_j$. 
Observe that $\phi_{ij}\psi_{jk}$ defines a section of $\mathfrak{gl}(E,\nabla)$
for which trace and determinant are constant; indeed, in a trivialization of $(E,\nabla)$ 
on $U_i\cap U_j\cap U_k$, the section $\phi_{ij}\psi_{jk}$
is constant. This is why $u_{ijk}$ is a section of the constant sheaf $\mathbb C$. 
One can show that this bilinear form is non degenerate, and defines a closed $2$-form $\omega_{\rm Con}$ on $\contotal$ which is symplectic: 
\[
d\omega_{\rm Con}\equiv0 \text{ and } \omega^n_{\rm Con}=\underbrace{\omega_{\rm Con}\wedge\cdots\wedge\omega_{\rm Con}}_{n\ \text{times}}\neq 0,
\]
i.e. $\omega^n_{\rm Con}$ defines a holomorphic volume form, a trivialization of the canonical bundle. We refer to \cite{Iwasaki91}
for details and proofs.

\begin{remark}
For the non generic resonant case
\[
d+\begin{pmatrix}\nu&0\\ 0&\nu+n\end{pmatrix}\frac{dx}{x},\ \ \ n\in\mathbb{Z}_{\ge0}
\]
the sequence \eqref{eq:shortexactnabla1} is not exact anymore: $\nabla^1$ is not surjective. 
This is why the description of the symplectic structure is more complicated in \cite{AL,IIS, Inaba,IS,Komyo}.
They use hypercohomology to overcome this difficulty. In our case, we want to understand the symplectic structure
at a generic point of the moduli space, so it is not necessary to consider this kind of models.
\end{remark}

In order to describe the symplectic structure $\omega_{\rm Con}$, we will make a reduction step. 
We will see that, locally, we can split $T_{(E,\nabla)}\contotal$ as a direct sum 
of two distinguished subspaces: one concerning Higgs fields 
and the other being the tangent space to a Lagrangian submanifold. 
This will allow us to compare $\omega_{\rm Con}$ with the natural symplectic form for Higgs fields.

Consider 
\[
{\rm Higgs}(E, {\bf p}) = H^0(\mathfrak{sl}(E,\textbf{p})\otimes\Omega^1_C(D)^{\text{nil}})
\]
the space of Higgs fields for $E$ with nilpotent residues preserving ${\bf p}$. We define a pairing 
\begin{equation}\label{eq:PerfectPairing1}
\begin{tikzcd}[row sep=1.pt, ampersand replacement=\&]
	{\rm Higgs}(E, {\bf p}) \times H^1(\mathfrak{sl}(E,\textbf{p})) \arrow[r, "\displaystyle \eta"] \& H^1(\Omega^1_C)\simeq \mathbb{C} \\
	(\Theta,\{\phi_{ij}\}) \arrow[r, maps to,shorten <= 3.8em] \& \{{\rm tr}(\Theta\cdot \phi_{ij}) \}
\end{tikzcd}
\end{equation}
One proves, see \cite{Iwasaki91}, that $\eta$ is a perfect pairing via identifying $H^1(\mathfrak{sl}(E,\textbf{p}))^\vee\simeq T^*_{(E,{\bf p})}\mathcal{B}$ 
and $\text{Higgs}(E,\textbf{p})$. 
In other words, pairing $\eta$ corresponds to the Liouville form on $T^*\mathcal{B}$.

Consider the following diagram
\begin{equation}\label{diag:partialSymplectic}
	\begin{tikzcd}[row sep=5.pt, column sep = 4.5em]
		& H^1(\mathfrak{sl}(E,\nabla))\times H^1(\mathfrak{sl}(E,\nabla)) \arrow[r, "\displaystyle\omega_{\rm Con}" ] & H^2(\mathbb C) \\
	\text{Higgs}(E,\textbf{p})	\times H^1(\mathfrak{sl}(E,\nabla)) \arrow[dr] \arrow[ur] && \\
		&	\text{Higgs}(E,\textbf{p})\times H^1(\mathfrak{sl}(E,\textbf{p}))\arrow[r,"\displaystyle \eta"] & H^1(\Omega^1)\arrow[uu, "\displaystyle\sim" labl, swap, "\displaystyle \delta"]
	\end{tikzcd}
\end{equation}
where we have used the maps from \eqref{eq:extractlongnabla1} and $\delta$ is the connecting morphism coming from the long exact sequence associated to 
\[
0\longrightarrow\mathbb{C} \longrightarrow \mathcal{O}_C\stackrel{d}{\longrightarrow}\Omega^1_C\longrightarrow 0.
\]
Then we reconcile $\omega_{\rm Con}$ and the pairing $\eta$.

\begin{lemma}\label{Lem:CommutativePartDiag} 
The diagram \eqref{diag:partialSymplectic} is commutative.
\end{lemma}

\begin{proof}Let $(\theta,\{\psi_{ij}\})$ be an element of $\text{Higgs}(E,\textbf{p})	\times H^1(\mathfrak{sl}(E,\nabla)) $. As $\theta$ is a parabolic Higgs field,
one can locally write $\theta\vert_{U_i}=\nabla^1(\varphi_i)$ and the upper map yields 
$(\{\varphi_j-\varphi_i\},\{\psi_{ij}\}) \in H^1(\mathfrak{sl}(E,\nabla))^2$, which is sent to $\{\text{tr}((\varphi_j-\varphi_i)\circ\psi_{jk})\}\in H^2(\mathbb C)$. On the other hand, by first going down, the pairing $\eta$ from \eqref{eq:PerfectPairing1} 
gives $\{\text{tr}(\theta\circ\psi_{ij})\}\in H^1(\Omega^1)$. To compute the image by $\delta$, we first remark that we have
\begin{align*}
\text{tr}(\theta\circ\psi_{ij})&= \text{tr}(\nabla^1(\varphi_i) \circ \psi_{ij}) =\\
& = \text{tr}(\nabla^1(\varphi_i \circ \psi_{ij}))-\text{tr}(\varphi_i \circ \nabla^1(\psi_{ij})) = \\
&= \text{tr}(\nabla^1(\varphi_i \circ \psi_{ij})) =\\
&= d(\text{tr}(\varphi_i \circ \psi_{ij}))
\end{align*}

Here, we have used that $\nabla^1(\psi_{ij}) = 0$ and the properties $\nabla^1(F\circ G)=\nabla^1(F)\circ G+F\circ \nabla^1(G)$ and 
$\text{tr}(\nabla^1(F))=d(\text{tr}(F))$, which can be deduced from a local computation in matrix form.
Therefore, we have 
\begin{align*}
	\delta(\{\text{tr}(\theta\circ\psi_{ij})\}) & = \{\text{tr}(\varphi_i \circ \psi_{ij})+\text{tr}(\varphi_j \circ \psi_{jk})-\text{tr}(\varphi_i \circ \psi_{ik})\}=\\
	& = \{\text{tr}(\varphi_i \circ (\psi_{ij}-\psi_{ik}))+\text{tr}(\varphi_j \circ \psi_{jk})\}= \\
	& = \{\text{tr}((\varphi_j-\varphi_i)\circ\psi_{jk})\}
\end{align*}
and get the same result.
\end{proof}

Given a choice of directions $\epsilon \in \{+.-\}^n$, let $\mathcal{C}$ denote the open subset of $\contotal$ composed by the connections $(E,\nabla)$ such that $(E, {\bf p}^\epsilon (\nabla))$ belongs to $\mathcal{B}$. Recall that a subvariety is called Lagrangian the restriction of symplectic form vanishes identically.

Assume now that we are given a local section $\nabla_0\colon B\rightarrow \mathcal{C}$ on an open subset $B\subset\mathcal{B}$.
Then we can define an isomorphism to the moduli space of parabolic Higgs bundles
\[
\Psi_{\nabla_0} \colon \left.\mathcal{C}\right|_B \longrightarrow\left.{\rm Higgs}\right|_B 
\]
by setting $(E,\nabla)\mapsto(E,\nabla-\nabla_0(E,\textbf{p}))$; the inverse being given by $(E,\varphi)\mapsto(E,\nabla_0(E,\textbf{p})+\varphi)$. 
Recall that \eqref{eq:PerfectPairing1} identifies Higgs fields and the cotangent space to $\mathcal{B}$. This yields an identification between the moduli space of parabolic Higgs bundles and the cotangent bundle $T^*\mathcal{B}$ which, itself, carries the Liouville symplectic structure that we will call $\omega_{\text{Higgs}}$.

Recall that $\Psi_{\nabla_0}$ is a symplectomorphism if $\Psi^*_{\nabla_0}\omega_{\text{Higgs}}= \omega_{\text{Con}}$. On the other hand, we say that the section $\nabla_0$ is Lagrangian if so is its image.

\begin{prop}\label{prop:LagrangianReduction}
The isomorphism $\Psi_{\nabla_0}$ is a symplectomorphism if and only if the section $\nabla_0\colon B\longrightarrow\mathcal{C}$ is Lagrangian.
\end{prop}

Before proving this proposition, let us recall some basic property of the Liouville form $\omega_{\text{Liouville}}$ on $T^*M$
for an arbitrary manifold $M$. If $(p_1,\ldots,p_n)$ are coordinates on $M$, and if we denote by $(q_1,\ldots,q_n)$
the coordinates on the fibers such that $dp_i$ corresponds to the section $q_i\equiv1$ and $q_j\equiv0$ for $j\not=i$,
then $\omega_{\text{Liouville}}=\sum_{i=1}^ndp_i\wedge dq_i$.
Given a section $\alpha\colon M\longrightarrow T^*M$, we can define an affine bundle transformation by
\[
\psi_\alpha \colon T^*M\longrightarrow T^*M ;\ (p,q)\longmapsto (p,q+\alpha(p)),
\]
where $T^*M$ is viewed as an affine bundle.

\begin{lemma}\label{lem:symplagranclosed}
Let $M$, $\alpha$ and $\psi_\alpha$ as above. Then the following are equivalent:
\begin{itemize}
\item $\psi_\alpha$ is a symplectomorphism,
\item $\alpha(M)$ is Lagrangian (as a submanifold),
\item $d\alpha=0$ (i.e. closed as a $1$-form).
\end{itemize}
\end{lemma}

\begin{proof}[Proof of Proposition \ref{prop:LagrangianReduction}]
We want to prove that $\Psi_{\nabla_0}$ is a symplectomorphism; then we only need to show that $\Psi^*_{\nabla_0}\omega_{\text{Higgs}}= \omega_{\text{Con}}$ at an arbitrary point $(E,\varphi)\in\text{Higgs}$.
As the vector bundle $\text{Higgs}$ identifies by \eqref{eq:PerfectPairing1} to $T^*\mathcal{B}$, we can use 
Lemma \ref{lem:symplagranclosed} to compose with an affine transformation sending $\varphi$ to 
the zero Higgs field. Now, the tangent space decomposes as 
\[
T_{(E,0)}\text{Higgs}=\text{Higgs}(E,\textbf{p}) \oplus T_{(E,\textbf{p})}\mathcal{B}
\]
and the symplectic form $\omega_{\text{Higgs}}$ may be written as 
\[
\omega_{\text{Higgs}}(u_1+v_1,u_2+v_2) = \eta(u_1,v_2)- \eta(u_2,v_1)
\]
where $\eta$ is the pairing given by \eqref{eq:PerfectPairing1}. Now, the tangent space at $(E,\nabla)$ decomposes as 
\[
T_{(E,\nabla)}\mathcal{C}=\text{Higgs}(E,\textbf{p}) \oplus T_{(E,\nabla)}\nabla_0(B) = \text{Higgs}(E,\textbf{p}) \oplus D\nabla_0(T_{(E,\textbf{p})}{\mathcal{B}}) 
\]
and note that the differential of $\Psi_{\nabla_0}^{-1}$ is the block diagonal $D\Psi_{\nabla_0}^{-1} = id \oplus D\nabla_0$. Now $\omega_{\text{Con}}$ may be written as
\begin{align*}
\omega_{\rm Con}\!\left(u_1+D\nabla_0v_1,u_2+D\nabla_0v_2\right) &= \omega_{\rm Con}\!\left(u_1,D\nabla_0v_2\right)- \omega_{\rm Con}\!\left(u_2,D\nabla_0v_1\right) + \\
& \quad\quad + \omega_{\rm Con}(u_1,u_2) + \omega_{\rm Con}\!\left(D\nabla_0v_1,D\nabla_0v_2\right)\\
& = \eta\!\left(u_1,D\nabla_0v_2\right)- \eta\!\left(u_2,D\nabla_0v_1\right) + \omega_{\rm Con}\!\left(D\nabla_0v_1,D\nabla_0v_2\right) \\
& = \Psi^*_{\nabla_0}\omega_{\text{Higgs}}\!\left(u_1+v_1,u_2+v_2\right) + \omega_{\rm Con}\!\left(D\nabla_0v_1,D\nabla_0v_2\right) 
\end{align*}
Indeed, it follows from the diagram \eqref{diag:partialSymplectic}. To prove the vanishing of $\omega_{\rm Con}(u_1,u_2)$, note that by \eqref{eq:extractlongnabla1} the image of $\text{Higgs}(E,\textbf{p})$ in $H^1(\mathfrak{sl}(E,\textbf{p}))$ is zero. Hence by taking the lower path in \eqref{diag:partialSymplectic} 
the pairing $\eta$ yields $0\in H^1(\Omega^1)$. 

Therefore $\Psi^*_{\nabla_0}\omega_{\text{Higgs}}= \omega_{\text{Con}}$ if and only if $\omega_{\rm Con}\!\left(D\nabla_0v_1,D\nabla_0v_2\right) = 0$, i.e. $\nabla_0(B)$ is Lagrangian.
\end{proof}

In our situation, we can compute this symplectic structure explicitly on a Zariski open set. Indeed, let $\mathcal{C} \rightarrow\mathcal{B}$ be given by the choice of directions $\epsilon = (\epsilon_1, \dots, \epsilon_n) \in \{-, + \}^n$ and let $\delta = -\epsilon$ be the opposite choice. One can verify that a fiber of $(E,\nabla)\mapsto (E, {\bf p}^\delta(\nabla))$ is a rational section of $(E,\nabla)\mapsto (E, {\bf p}^\epsilon(\nabla))$. The reduction is symplectic if the chosen section is Lagrangian, wherever it is defined. This is ensured if the projection is Lagrangian, which is covered by the following result.

\begin{lemma}The forgetful map $\mathcal{C}\longrightarrow \mathcal{B} ;\ (E,\nabla)\mapsto (E,\textbf{p}^\epsilon(\nabla))$ is Lagrangian, i.e. its fibers are Lagrangian subvarieties.
\end{lemma}

\begin{proof} The differential of the forgetful map at $(E,\nabla)$ determines the dual to the map $\text{Higgs}(E,\textbf{p}) \longrightarrow H^1(\mathfrak{sl}(E,\nabla))$ from \eqref{eq:extractlongnabla1}. Then we want to prove that the pairing $\omega_{\rm Con}$ of \eqref{eq:DefSymplecticInfinitesimal}
is zero on a pair $(\{\phi_{ij}\},\{\psi_{ij}\})$ where each entry is in the image of $\text{Higgs}(E,\textbf{p})$. We showed that this is true in the proof of the last proposition.
\end{proof}

We now apply this construction to exhibit the symplectic form on ${\rm Con}^\nu$ via explicit computations with Fuchsian systems.

\subsection{Computation of the symplectic structure}\label{subsec:sympfuchs} 

 We have seen that choosing a Lagrangian section $\nabla_0$ of $\pi_\epsilon \colon {\rm Con}^{\nu}\rightarrow {\rm Bun}$
allows us to reduce the computation of the symplectic structure of $\mathcal {\rm Con}^{\nu}$ to that of Higgs moduli space 
(see Proposition \ref{prop:LagrangianReduction}). We can apply this to the setting of Fuchsian systems.
The section $\nabla_0$ constructed in section \ref{section: U_0} corresponds, via the isomorphism of Proposition \ref{prop:consys},
to the fiber of $\pi_{-\epsilon}$ over the point $(\infty,\ldots,\infty)\in\mathrm{Bun}\simeq(\mathbb P^1)^n$;
it is therefore Lagrangian. Let us compute the symplectic structure on the moduli space of Higgs bundles.
For this, we consider the chart:
\[
\mathbb{C}^{2n}\longrightarrow\mathrm{Higgs}(C,D) ;\ (z_1,\ldots,z_n,r_1,\ldots,r_n)\longmapsto (E_1,\Theta)
\]
where $\Theta$ is obtained by applying elementary transformations at $w_0$, $w_1$, $w_\lambda$ to 
\[
r_1\Theta^0_1+\cdots+r_n\Theta^0_n
\]
as defined in section \ref{section: U_0}. The image coincides with the restriction $\mathrm{Higgs}(C,D)\vert_{U_0}$ of Higgs fields
on parabolic bundles belonging to the chart $U_0\subset\mathrm{Bun}$ of section \ref{section: U_0}. We want now to compute the isomorphism
$\mathrm{Higgs}(C,D)\vert_{U_0}\simeq T^*U_0$ so that the symplectic form can be easily deduced from the Liouville form.

Here we follow the ideas of the proof of \cite[Proposition 6.1]{ViktoriaFrank}.
Given a parabolic vector bundle $(E_1, {\bf p})$ in $U_0$, or equivalently $(\mathcal{O}_C\oplus\mathcal{O}_C, {\bf p'})$
defined by $(z_1,\ldots,z_n)\in\mathbb C^n$, Serre Duality \eqref{eq:PerfectPairing1} provides a perfect pairing 
\[
\left<\cdot,\cdot\right> \colon {\rm Higgs}(E, {\bf p}) \times H^1(\mathfrak{sl}(E,\textbf{p}))\longrightarrow\mathbb{C}
\]
which allows us to identify the fiber of ${\rm Higgs}$ at $(E_1, {\bf p})$ with the dual of the tangent space
\[
T_{(E_1, {\bf p})}U_0\simeq H^1(\mathfrak{sl}(E_1,\textbf{p})).
\]
Let us compute first this latter identification. 

The deformation $(z_1,\ldots,z_j+s,\ldots,z_n)$ of the $j^{\text{th}}$ component of the parabolic structure
can be realized as follows. Cut out the curve $C$ into two open sets $U_j=C-\{t_j\}$ and $V_j$ a small disk around $t_j$.
Then we can paste the restrictions of $(\mathcal{O}_C\oplus\mathcal{O}_C, {\bf p'})$ to these open sets by means of the 
following gauge transformation over the restriction
\[
\left.\left(\mathcal{O}_C\oplus\mathcal{O}_C, {\bf p'}\right)\right\vert_{U_j\cap V_j}\longrightarrow \left.\left(\mathcal{O}_C\oplus\mathcal{O}_C, {\bf p'}\right)\right\vert_{U_j\cap V_j} ;\ 
\begin{pmatrix}z\\ w\end{pmatrix} \longmapsto \begin{pmatrix}1&s\\ 0&1\end{pmatrix}\begin{pmatrix}z\\ w\end{pmatrix}.
\]
This clearly shifts the parabolic direction $(z_j:1)\mapsto(z_j+s:1)$ without moving other parabolic directions.
After derivating, we obtain that the infinitesimal deformation $\frac{\partial}{\partial z_j}$ corresponds to the 
gauge transformation $\begin{pmatrix}0&1\\ 0&0\end{pmatrix}\in H^0(U_j\cup V_j,\mathfrak{sl}(\mathcal{O}_C\oplus\mathcal{O}_C, {\bf p'}))$,
which can better be viewed as a cocyle with respect to the covering $(U_j,V_j)$.
We can now compute the Serre Duality: the pairing of the Higgs field $\Theta^0_i$ with the above cocycle is defined by 
\[
\mathrm{tr}\left(\Theta^0_k\cdot\begin{pmatrix}0&1\\ 0&0\end{pmatrix}\right)=-z_k\phi_0+\phi_1+\frac{y_k}{x_k}\omega+\theta_k
\]
a holomorphic $1$-form on $U_j\cap V_j$. Viewed as a cocycle in $H^1(\Omega)$, one can give a meromorphic resolution 
by considering it as the difference between the zero $1$-form on $U_j$ and the meromorphic $1$-form $-z_k\phi_0+\phi_1+\frac{y_k}{x_k}\omega+\theta_k$ 
on $V_j$. The residue at $t_j$ gives the pairing: 
\[
\left< \Theta^0_k , \begin{pmatrix}0&1\\ 0&0\end{pmatrix} \right>=\mathrm{Res}_{t_j}\!\left(-z_k\phi_0+\phi_1+\frac{y_k}{x_k}\omega+\theta_k\right)=\delta_j^k
\]
where $\delta_j^k$ is the Kronecker symbol. We deduce that $\Theta^0_j$ is dual to $\frac{\partial}{\partial z_j}$,
and the symplectic structure on $\mathrm{Higgs}\ni\sum_j r_j\Theta_j^0$ is given by 
\[
\omega_{\rm Higgs}=dz_1\wedge dr_1+\cdots+dz_n\wedge dr_n.
\]
Now, because our section $\nabla_0\colon U_0\rightarrow\mathrm{Con}$ is Lagrangian, we can apply Proposition \ref{prop:LagrangianReduction}
and deduce the symplectic structure on $\mathrm{Con}\ni\nabla_0+\sum_j r_j\Theta_j^0$, namely:
\begin{equation}\label{eq:Darboux0}
\omega_{\rm Con}=dz_1\wedge dr_1+\cdots+dz_n\wedge dr_n.
\end{equation}
The same construction shows, mutatis mutandis, that the symplectic form on $\nabla_\infty+\sum_j s_j\Theta_j^\infty$	is
\begin{equation}\label{eq:DarbouxInfty}
\omega_{\rm Con}= -(dw_1\wedge ds_1+\cdots+dw_n\wedge ds_n).	
\end{equation}

Next we will see how $\omega_{\rm Con}$ encodes the eigenvalue $\nu$. First we consider only maps that preserves the affine bundle structure. 

\begin{thm}\label{thm:symptorelli1}
	If there exists a bundle symplectic isomorphism
	\[
	\begin{tikzcd}[column sep = 0.5cm]
		\left(\mathcal {\rm Con}^{\nu},\omega\right) \arrow[dr, "\pi_{+}"'] \arrow[rr,"\Phi", "\sim"'] && \left(\mathcal {\rm Con}^{\tilde{\nu}},\tilde{\omega}\right) \arrow[dl, "\pi_{+}"] \\ & {\rm Bun} &
	\end{tikzcd}
	\]
	then $\tilde{\nu} = \nu$ and $\Phi$ is the identity.
\end{thm} 

\begin{proof}
Since $\pi_{+}\circ \Phi = \pi_{+}$, it preserves $\pi_+^{-1}(U_0)$ and $\pi_+^{-1}(U_\infty)$. Then we may write the map in coordinates. For ${\rm Con}^{\nu}$ we will consider the two coordinate patches: 
\[
(\pi_+^{-1}(U_0), (z,r)) \text{ and } (\pi_+^{-1}(U_\infty), (w,s)) 
\]
where $z_j = \frac{1}{w_j}$ and $r_j = s_j w_j^2 + \nu_jw_j$, see \eqref{eq:aftrans}. For ${\rm Con}^{\tilde{\nu}}$ we consider the corresponding open sets with coordinates $(z,\tilde{r})$ and $(w, \tilde{s})$, respectively, where $z_j = \frac{1}{w_j}$ and $\tilde{r}_j = \tilde{s}_j w_j^2 + \tilde{\nu}_jw_j$. In both cases, $\pi_+$ is the projection onto the first $n$ coordinates.
	
Therefore we can express $\Phi$ as
\begin{align*}
	\left.\Phi\right|_{\pi_+^{-1}(U_0)} & = \left(z_1, \cdots, z_n, \phi^0_1(z,r), \cdots, \phi^0_n(z,r)\right) = (z, \tilde{r}), \\
\left.\Phi\right|_{\pi_+^{-1}(U_\infty)} & = \left(w_1, \cdots, w_n, \phi^\infty_1(w,s), \cdots, \phi^\infty_n(w,s)\right) = (w, \tilde{s}),
\end{align*}
where $\phi^0_j(z,r)$ and $\phi^\infty_j(w,s)$ are holomorphic functions on $\mathbb{C}^{2n}$ satisfying the compatibility conditions:
\begin{equation}\label{eq:compatb}
	\phi^\infty_j(\cdots, w_k, \cdots , s_k, \cdots ) w_j^2 + \tilde{\nu_j} w_j = \phi^0_j\left(\cdots, \frac{1}{w_k}, \cdots, s_k w_k^2 + \nu_kw_k, \cdots \right), \quad \forall w \in (\mathbb{C}^\ast)^n.
\end{equation}

The condition that $\Phi$ is symplectic, $\Phi^\ast \tilde{\omega} = \omega$, translates to the coordinate charts via the $2$-forms in display \eqref{eq:Darboux0} and \eqref{eq:DarbouxInfty}. In particular, $\phi^0_j$ and $\phi^\infty_j$ must satisfy 
\[
\frac{\partial\phi^0_j }{\partial r_k} = \frac{\partial\phi^\infty_j }{\partial s_k} = \delta^j_k, 
\]
where $\delta^j_k$ is the Kronecker symbol. Therefore, there exist holomorphic functions $c_j^0(z)$ and $c_j^\infty(w)$ such that
\begin{align*}
	\phi^0_j(z,r) = r_j + c_j^0(z) \text{ and } \phi^\infty_j(w,s) = s_j + c_j^\infty(w).
\end{align*}
Substituting these expressions in \eqref{eq:compatb} yields
\begin{equation}\label{eq:compatb2}
	 c_j^0\left(\frac{1}{w_1} , \cdots, \frac{1}{w_n}\right) = c_j^\infty(w)w_j^2 + (\tilde{\nu_j} - \nu_j)w_j , \quad \forall w \in (\mathbb{C}^\ast)^n.
\end{equation}
It follows that $c_j^0$ extends as a holomorphic function on $U_0\cup U_\infty$. On the other hand, the complement of $U_0\cup U_\infty$ in $(\mathbb{P}^1)^n$ has codimension $2$, which implies that $c_j^0$ extends to the whole space by Hartogs's Theorem. Thus $c_j^0$ is constant and so is $c_j^\infty$, by the same argument.
Thus, the equation \eqref{eq:compatb2} implies that 
\[
c_j^0 = c_j^\infty = 0 \text{ and } \tilde{\nu_j} = \nu_j.
\]

\end{proof}

We can use the map $\Par \colon {\rm Con}^{\nu}\rightarrow S^n$ given by Theorem \ref{thm:parconSn} to push $\omega_{\rm Con}$ forward to $S^n$. The opposite parabolic structure of our universal family $\nabla = \nabla_0 + \sum_j r_j\Theta_j^0$ is given by
\[
(\zeta_j:1):=(z_jr_j-\nu_j:r_j).
\]
Therefore, if we denote by $(z_j,\zeta_j)$ the coordinates on the $j^{\text{th}}$ factor $S\simeq\mathbb P^1\times\mathbb P^1\setminus\{z_j=\zeta_j\}$, then we have
\[
\zeta_j=z_j-\frac{\nu_j}{r_j}\ \ \ \Leftrightarrow\ \ \ r_j=\frac{\nu_j}{z_j -\zeta_j}
\]
which yields after substitution:
\[
\omega_{\nu}= \sum_{i = 1}^{n}\nu_i\frac{dz_i\wedge d\zeta_i}{(z_i-\zeta_i)^2}.
\]

Using this identification we can prove the following corollary.

\begin{cor}\label{cor:symptorelli}
	If there exists fiber preserving symplectic isomorphism
	\[
	\begin{tikzcd}
		\left( {\rm Con}^{\nu},\omega\right) \arrow[d, "\pi_{+}"] \arrow[r,"\Phi", "\sim"'] & \left( {\rm Con}^{\tilde{\nu}},\tilde{\omega}\right) \arrow[d, "\pi_{+}"] \\ {\rm Bun} \arrow[r, "\phi", "\sim"']& {\rm Bun}
	\end{tikzcd}
	\]
	then there exists a permutation $\sigma$ of $n$ elements such that $\tilde{\nu}_k = \nu_{\sigma(k)}$ for every $k\in\{1,\cdots,n\}$.
\end{cor}

\begin{proof}
Note that the case where $\phi$ is the identity is ensured by Theorem \ref{thm:symptorelli1}. In particular, the identity on ${\rm Bun}$ has a unique lifting which is the identity on ${\rm Con}^{\nu}$. It follows that, a general $\Phi$ is determined by the map $\phi$ on the base. Indeed, if $\Phi_1$ and $\Phi_2$ lift $\phi$, then $\Phi_2 \circ \Phi_1^{-1}$ lifts the identity, hence $\Phi_2 = \Phi_1$. The proof of the statement will follow from a reduction to this case.
	
Up to composing with ${\rm Par}$, we may prove the result for 
\[
\Phi\colon (S^n, \omega_{\nu}) \longrightarrow (S^n, \omega_{\tilde{\nu}}).
\]

Any automorphism $\phi$ of $(\mathbb P^1)^n$ has the form 
\[
\phi(z) = \left(\varphi_1(z_{\sigma(1)}), \cdots, \varphi_n(z_{\sigma(n)})\right)
\]
where $\sigma$ is a permutation of $n$ elements and $\varphi_i$ are M\"obius transformations. Now define 
\[
\Xi(z,\zeta) = \left( \varphi_1(z_{\sigma(1)}), \cdots, \varphi_n(z_{\sigma(n)}), \varphi_1(\zeta_{\sigma(1)}), \cdots, \varphi_n(\zeta_{\sigma(n)})\right).
\]
We will show that $\Xi$ is an automorphism of $S^n$ such that $\Xi^\ast\omega_{\nu^\sigma}= \omega_{\nu}$, where $\nu_j ^\sigma = \nu_{\sigma(j)}$. Recall that any M\"obius transformation is a composition of basic transformations $t\mapsto \alpha t$, $t\mapsto t+1$ and $t\mapsto 1/t$. Therefore we can factorize each $\varphi_j$ to show that
\[
\left( \frac{d\varphi_j(z_{\sigma(j)})\wedge d\varphi_j(\zeta_{\sigma(j)})}{(\varphi_j(z_{\sigma(j)})-\varphi_j(\zeta_{\sigma(j)}))^2} \right) = \frac{dz_{\sigma(j)}\wedge d\zeta_{\sigma(j)}}{\left(z_{\sigma(j)}-\zeta_{\sigma(j)}\right)^2} 
\]
for all $j\in \{1,\cdots, n\}$, and this shows that $\Xi^*(\omega_{\nu}) = \omega_{\nu^\sigma}$. 

Now consider the $\Phi$ the unique extension of $\phi$ such that $\Phi^\ast \omega_{\tilde{\nu}} = \omega_{\nu}$. We have that $\Phi \circ \Xi^{-1}$ lifts the identity and $\left(\Phi \circ \Xi^{-1}\right)^\ast \! \omega_{\tilde{\nu}} = \omega_{\nu^\sigma}$. Thus the result follows from Theorem \ref{thm:symptorelli1}.
\end{proof}

\section{Apparent map}\label{sect:appmap}

Given a connection $\nabla\colon E_1\rightarrow E_1\otimes \Omega^1_C(D)$ we can define an $\mathcal O_C$-linear map:
\[
\varphi_{\nabla}\colon \mathcal O_C \stackrel{s}{\hookrightarrow} E_1 \stackrel{\nabla}{\longrightarrow} E_1\otimes\Omega_C^1(D) \longrightarrow (E_1/\mathcal O_C)\otimes \Omega_C^1(D)
\]
where the last arrow is defined by the quotient map from $E_1$ to $E_1/s(\mathcal O_C)\simeq \mathcal O_C(w_{\infty})$ (denoted $E_1/\mathcal O_C$ for short). The zero divisor of $\varphi_{\nabla}$ defines an element of the linear system 
\[
Z(\varphi_{\nabla})\in|\mathcal O_C(w_{\infty}+D)|.
\]
Since $\deg D=n$ and $\Omega_C^1\simeq \mathcal O_C$, as $C$ has genus one, then 
\[
|\mathcal O_C(w_{\infty}+D)| = \mathbb P\!\left(H^0(C,(E_1/\mathcal O_C)\otimes \Omega_C^1(D))\right)\simeq \mathbb P^n.
\]
Hence we may define a rational map
\[
\operatorname{App}\colon {\rm Con}^{\nu} \dashrightarrow |\mathcal O_C(w_{\infty}+D)|
\]
which associates $(E_1, \nabla)$ to $Z(\varphi_{\nabla})$. Under a generic hypothesis on the eigenvalues we show that $\operatorname{App}$ is in fact a morphism. It turns out that this hypothesis is also necessary as we prove in the following lemma.

\begin{lemma}\label{lem:appmor}
The rational map $\operatorname{App}\colon {\rm Con}^{\nu} \dashrightarrow |\mathcal O_C(w_{\infty}+D)|$ is a morphism if and only if $\nu_1^{\epsilon_1}+\cdots +\nu_n^{\epsilon_n} \neq 0$ for any $\epsilon_{k}\in \{+,-\}$.
\end{lemma}

\begin{proof}
The indeterminacy locus of $\operatorname{App}$ is composed by the connections $\nabla$ mapped to zero in $(E_1/\mathcal O_C)\otimes \Omega_C^1(D)$ which means that there exists a meromorphic $1$-form $\xi$ with poles at most on $D$ such that 
\[
\nabla(s) = s\xi.
\]
In such case we have for each $j=1, \dots, n$ that $s(t_j) \in p_j^{\epsilon_j}(\nabla)$ with $\epsilon_{j}\in \{+,-\}$. And we also have that $\nabla$ restricts to a logarithmic connection on $\mathcal O_C$ whence
\[
\nu_1^{\epsilon_1}+\cdots +\nu_n^{\epsilon_n} = 0.
\]

Conversely, suppose that there exist $\delta_j \in \{+,-\}$, $j=1, \dots n$, such that $\nu_1^{\delta_1}+\cdots +\nu_n^{\delta_n} = 0$. On the one hand there exists a meromorphic $1$-form $\xi$ with simple poles only on $D$ and prescribed residues 
\[
Res_{t_j}\xi = \nu_j^{\delta_j}.
\]
On the other hand let $\nabla$ be a logarithmic connection such that $s(t_j) \in p_j^{\delta_j}(\nabla)$. It follows that $\nabla(s) - s\xi$ is a holomorphic section of $E_1 \otimes \Omega_C^1$. Since $ H^0(C, E_1 \otimes \Omega_C^1)$ is generated by $s\omega$ (recall that $\omega$ is a holomorphic one-form on $C$) it follows that
\[
\nabla(s) = s\,(\xi + c\,\omega)
\]
for some $c\in \mathbb{C}$. Hence $\varphi_{\nabla} = 0$.
\end{proof}

We may also consider the forgetful map 
\[
\pi_+ \colon {\rm Con}^{\nu} \longrightarrow {\rm Bun} \simeq (\mathbb P^1)^n
\]
which sends $(E_1, \nabla)$ to the parabolic vector bundle $(E_1, {\bf p}^+(\nabla))$. We also denote $\pi_+$ the natural extension of this map to $\overline{{Con}^{\nu}}$. Combining these two maps we can give a birational model for $\overline{{Con}^{\nu}}$ as we prove in the following result.

\begin{thm}\label{thm:apparentbirat}
If $\nu_1^{\epsilon_1}+\cdots +\nu_n^{\epsilon_n} \neq 0$ for any $\epsilon_{k}\in \{+,-\}$ then the map $\pi_+\times \operatorname{App}$ induces a birational map
\[
\pi_+ \times \operatorname{App} \colon \overline{{Con}^{\nu}} \dashrightarrow {\rm Bun}\times |\mathcal O_C(w_{\infty}+D)|
\]
whose indeterminacy locus is contained in $\overline{{Con}^{\nu}} \backslash {Con}^{\nu}$. Moreover, given $(E,{\bf{p}})\in \BB$ the rank of
\[
\left.(\pi_+\times \operatorname{App})\right|_{\pi_+^{-1}(E,{\bf{p}})}\colon \pi_+^{-1}(E,{\bf{p}}) \longrightarrow |\mathcal O_C(w_{\infty}+D)|
\]
coincides with the cardinality of the set $\{j\mid p_j\not\subset \mathcal O_{C} \}$.
\end{thm}

\begin{proof}
Let $s$ denote the section $\mathcal{O}_C \hookrightarrow E_1$. For any logarithmic connection $\nabla$ on $E_1$ with poles on $D=t_1 + \dots + t_n$ we have that $\operatorname{App}(\nabla)$ is given by the zero divisor of
\[
\varphi_\nabla = \nabla(s) \wedge s \in H^0(C, \det(E_1)\otimes\Omega_C(D))
\]
which is the same as defining
\[
\operatorname{App}(\nabla) = [\varphi_{\nabla}] \in \mathbb{P}\!\left( H^0(C, \det(E_1)\otimes\Omega_C(D))\right).
\]
By Lemma \ref{lem:appmor} the condition on the eigenvalues implies that $\operatorname{App}$, hence $\operatorname{Bun}\times \operatorname{App}$, is regular on ${Con}^{\nu}$. Therefore the indeterminacy locus must lie only in the boundary divisor.

Now fix $(E_1,{\bf{p}}=\{p_1,\dots, p_n\})\in {\rm Bun}$. We will compute the rank of $\left.\operatorname{App} \right|_{\pi_+^{-1}(E_1,{\bf{p}})}$. Consider $ H^0(\mathfrak{sl}(E_1)\otimes \Omega_C(t_j))$ the space of traceless Higgs fields with simple pole on $t_j$. Since its dimension is three, there exists a unique (up to scalar multiplication) strongly parabolic Higgs field $\Theta^p_j$ with respect to $\bf{p}$, i.e. $Res_{t_j}(\Theta^p_j)$ is nilpotent with image $p_j$ and has no other poles. Let $\nabla_0$ be a connection in $\pi_+^{-1}(E_1,{\bf{p}})$. Then any other $\lambda$-connection $\nabla \in \pi_+^{-1}(E_1,{\bf{p}})$ can be written in a unique way as
\[
\nabla = c_0\nabla_0 + \sum_{j=1}^n c_j\Theta^p_j,
\]
for some $(c_0:\dots : c_n)\in \mathbb{P}^n$. Hence 
\[
\varphi_{\nabla} = c_0 \nabla_0(s)\wedge s + \sum_{j=1}^nc_j\Theta^p_j(s) \wedge s.
\]

Suppose that $\Theta^p_j(s) \wedge s \neq 0$ for each $j$. Since they have one pole each and at different points these sections are linearly independent. Since $ H^0(C, \mathcal{O}_C(w_\infty + D))$ has dimension $n+1$ they form a basis together with $\nabla_0(s)\wedge s$ and the Apparent map restricted to this fiber is an isomorphism. Indeed, if 
\[
c_0 \nabla_0(s)\wedge s + \sum_{j=1}^nc_j\Theta^p_j(s) \wedge s = \left(c_0 \nabla_0(s) + \sum_{j=1}^nc_j\Theta^p_j(s) \right)\wedge s = 0
\]
then, by Lemma \ref{lem:appmor}, $c_0=0$ and it follows that $c_1 = \dots = c_n=0$ since the $\Theta^p_j(s) \wedge s$ are linearly independent.

Now suppose that there exists $j\in \{1,\dots,n\}$ such that $\Theta^p_j(s) \wedge s = 0$. This occurs if and only if $\Theta^p_j(s)$ is holomorphic at $t_j$ which in turn means that $s(t_j)\in p_j$. Indeed, since $H^0(C, E_1\otimes \Omega_C)$ is generated by $s\omega$, we know that $\Theta^p_j(s)$ is holomorphic at $t_j$ if and only if $\Theta^p_j(s) = c\,s\omega$ for some constant $c\in \mathbb{C}$. On the other hand, it is clear that ${\rm Res}_{t_j}\!\left(\Theta^p_j(s)\right)\in p_j$ is zero if and only if $s(t_j) \in p_j$.

Henceforth we conclude that, in general, the image of $\left.\operatorname{App} \right|_{\pi_+^{-1}(p)}$ is the linear space spanned by $\nabla_0(s)\wedge s$ and $\Theta^p_j(s) \wedge s$ such that $s(t_j) \not \in p_j$. In particular, the (projective) rank is given by
\[
\rk \left.\operatorname{App} \right|_{\pi_+^{-1}(E_1,{\bf{p}})} = \#\{j \mid s(t_j)\not\in p_j\}.
\]

\end{proof}

\begin{remark}
The Apparent map may be computed, explicitly, via the identification with Fuchsian systems. In this case, the section $\mathcal{O}_C \hookrightarrow E_1$ becomes 
\[
\mathcal{O}_C(w_\infty -w_0 -w_1 -w_\lambda) \hookrightarrow \mathcal{O}_C\oplus\mathcal{O}_C,
\]
given by multiplication with $\left( \frac{x-\lambda}{y}, \frac{(1-\lambda)x}{y}\right)$, and we can use the local universal connections over $U_0$ and $U_\infty$ to express ${\rm App}$. Following this path, one can give an alternative proof of Theorem \ref{thm:apparentbirat}. 
\end{remark}


\end{document}